\documentclass[a4paper]{article}
\usepackage{amssymb}
\usepackage{amsthm}
\usepackage{latexsym}
\usepackage{amsmath}
\usepackage{color}

\usepackage{comment}
\excludecomment{discuss}

\usepackage[margin=1in]{geometry} 
\usepackage{setspace} 

\newif\ifcol
\colfalse
\ifcol
\newcommand{\colorr}{\color[rgb]{0.8,0,0}}

\newcommand{\colorp}{\color[rgb]{0.8,0,0.8}}

\else
\newcommand{\colorr}{\color{black}}
\newcommand{\colorp}{\color{black}}

\fi

\newtheorem{definition}{Definition}[section]
\newtheorem{lemma}{Lemma}[section]
\newtheorem{proposition}{Proposition}[section]
\newtheorem{theorem}{Theorem}[section]
\newtheorem{remark}{Remark}[section]
\newtheorem{example}{Example}[section]
\newtheorem{corollary}{Corollary}[section]

\makeatletter
\@addtoreset{equation}{section}
\makeatother

\newcommand{\mca}{\mathcal{A}}
\newcommand{\mcb}{\mathcal{B}}

\newcommand{\mcf}{\mathcal{F}}

\newcommand{\mcl}{\mathcal{L}}

\newcommand{\mct}{\mathcal{T}}
\newcommand{\mcu}{\mathcal{U}}
\newcommand{\mcx}{\mathcal{X}}

\newcommand{\mcz}{\mathcal{Z}}

\newcommand{\mbbn}{\mathbb{N}}

\newcommand{\mbbr}{\mathbb{R}}

\newcommand{\al}{\alpha}
\newcommand{\del}{\delta}
\newcommand{\ep}{\epsilon}

\newcommand{\D}{\Delta}

\newcommand{\p}{\partial}

\def\sumj{\sum_{j=1}^{n}}

\def\supp{\mathrm{supp}}

\newcommand{\EQ}[1]{\begin{equation}{#1}\end{equation}}
\newcommand{\EQQ}[1]{\begin{equation}{#1 \nonumber}\end{equation}}
\newcommand{\EQN}[1]{\begin{equation}\begin{split}{#1}\end{split}\end{equation}}
\newcommand{\EQNN}[1]{\begin{equation}\begin{split}{#1 \nonumber}\end{split}\end{equation}}
\newcommand{\MAT}[2]{\left(\begin{array}{#1} #2 \end{array} \right)} 
\newcommand{\DIV}[2]{\left\{\begin{array}{#1} #2 \end{array} \right.} 

\newcommand{\PT}{\partial_\theta}
\newcommand{\PTT}{\partial_t}
\newcommand{\PS}{\partial_\sigma}
\newcommand{\PA}{\partial_\alpha}

\newcommand{\DELP}[2]{\frac{\partial_#2 p_{j,t}^{#1}}{p_{j,t}^{#1}}}
\newcommand{\DELPP}[2]{\frac{\partial_#2 p_{j,0}^{#1}}{p_{j,0}^{#1}}}
\newcommand{\DELPPT}[2]{\frac{#1 p_{j,0}^0(#2 p_{j,0}^0)^\top}{(p_{j,0}^0)^2}}
\newcommand{\DELTP}[1]{\frac{#1 \tilde{p}_j}{\tilde{p}_j}}
\newcommand{\DELPY}[1]{\frac{#1 \tilde{f}_t}{\tilde{f}_t}}
\newcommand{\DELPYY}[1]{\frac{#1 \tilde{f}_0}{\tilde{f}_0}}
\newcommand{\DELPYYT}{\frac{\PT \tilde{f}_0\PT \tilde{f}_0^\top}{\tilde{f}_0^2}}
\newcommand{\SUMJ}{\sum_{j=1}^n}
\newcommand{\INTTJ}{\int_{t_{j-1}}^{t_j}}
\newcommand{\PCB}{p_{x_{j-1},\alpha_{th}}^{c,\tau-t_{j-1}}(x)}
\newcommand{\PCA}{p_{x+y,\alpha_{th}}^{c,t_j-\tau}(x_j)}

\newcommand{\TOP}{\to^{\tilde{P}_{\alpha_0,n}}}
\newcommand{\XJ}{X_{t_{j-1}}}

\newcommand{\INTTP}[2]{\int #2 \tilde{p}_{j,#1}dx_j}
\newcommand{\BIGINTTP}[3]{\int\bigg| #2 \bigg|^{#1}#3\tilde{p}_j(\alpha_{th})dx_j}
\newcommand{\XC}{X^{\alpha,c}}

\begin{document}

\title{Local Asymptotic Mixed Normality via Transition Density Approximation and an Application to Ergodic Jump-Diffusion Processes}
\author{Teppei Ogihara$^*$ and Yuma Uehara$^{**}$\\
$*$
\begin{small}Graduate School of Information Science and Technology, University of Tokyo, \end{small}\\
\begin{small}7-3-1 Hongo, Bunkyo-ku, Tokyo 113--8656, Japan \end{small}\\
$**$
\begin{small} Department of Mathematics, Faculty of Engineering Science, Kansai University\end{small}\\
\begin{small}3-3-35 Yamate-cho, Suita-shi, Osaka 564-8680, Japan\end{small}
}
\maketitle

\noindent
{\bf Abstract.}
We study sufficient conditions for local asymptotic mixed normality. We weaken the sufficient conditions in Theorem~1 of Jeganathan (Sankhy\=a Ser.\ A 1982) so that they can be applied to a {\colorp wider class of statistical models including a jump-diffusion model.}
Moreover, we show that local asymptotic mixed normality of a statistical model generated by approximated transition density functions is implied for the original model. Together with density approximation by means of thresholding techniques, we show local asymptotic normality for a statistical model of discretely observed jump-diffusion processes where the drift coefficient, diffusion coefficient, and jump structure are parametrized.
As a consequence, the quasi-maximum-likelihood and Bayes-type estimators proposed in Shimizu and Yoshida (Stat.\ Inference Stoch.\ Process.\ 2006) and Ogihara and Yoshida (Stat.\ Inference Stoch.\ Process.\ 2011) are shown to be asymptotically efficient in this model. {\colorp Moreover, we can construct asymptotically uniformly most powerful tests for the parameters.}\\


\section{Introduction}

Local asymptotic normality (LAN) is an important property in asymptotic statistical theory because it enables us to discuss asymptotic efficiency of parameter estimators for parametric models.
H\'ajek \cite{haj70, haj72} showed the convolution theorem and the minimax theorem for statistical models that satisfy the LAN property.
Both theorems give different concepts of asymptotic efficiency.
The LAN property has mainly been studied for statistical models of independent observations.
Thereafter, this property has been extended to local asymptotic mixed normality (LAMN) so that we can address a wider class of statistical models.
Jeganathan~\cite{jeg82} showed the convolution theorem and the minimax theorem under the LAMN property.
LAN and LAMN enable several studies
of statistical methods, in addition to the efficiency of estimators. Several works have studied the construction of asymptotically uniformly most powerful tests
under LAN or LAMN (see i.g. Choi, Hall, and Schick~\cite{cho-etal96} and Basawa and Scott~\cite{bas-sco83}).
Moreover, Eguchi and Masuda~\cite{egu-mas18} studied the model selection problem
via Schwartz-type Bayesian information criteria (BIC), and showed model selection consistency of the BIC when the statistical model is locally asymptotically quadratic (which includes the case of LAMN).

For statistical models of discrete observations of semimartingales,
Gobet~\cite{gob01, gob02} showed the LAN and the LAMN properties for diffusion processes in the high-frequency limit of observations on a fixed interval and on a growing observation window, respectively. Related to processes with jumps,
A{\"{\i}}t-Sahalia and Jacod~\cite{ait-jac07} showed the LAN property
for some classes of L\'evy processes, including  symmetric stable processes, Kawai and Masuda~\cite{kaw-mas13} showed the LAN property for normal inverse Gaussian L\'evy processes, and Cl\'ement and Gloter \cite{cle-glo15} proved the LAMN property for a stochastic differential equation driven by a pure jump L\'evy process whose L\'evy measure is an $\alpha$-stable Le\'vy measure near the origin.
Statistical models of jump-diffusion processes were also studied in several papers; Kohatsu, Nualart, and Tran~\cite{koh-etal17} showed the LAN property for ergodic jump-diffusion processes whose drift coefficient depends on an unknown parameter, and Cl\'ement, Delattre, and Gloter \cite{cle-etal14} studied the LAMN property for the stochastic differential equations with jumps when the unknown parameter determines the jump structure and the jump times are deterministic and given.
Jump-diffusion processes are used for modeling various stochastic phenomena in many areas, such as econometrics, physics, and neuroscience.
Among the vast literature, we refer the reader to Rao~\cite{rao99} and Cont and Tankov~\cite{con-tan04} and references therein.
To our knowledge, there are no studies that show the LAN property for jump-diffusion processes with the drift coefficient, the diffusion coefficient, and the jump structure all parametrized, and in this paper, we focus on such a situation.

In the proofs of the LAN properties for diffusion processes of Gobet~\cite{gob01, gob02}, it is crucial that transition density functions satisfy estimates from above and below by Gaussian density functions up to constants: so-called Aronson-type estimates.
Unlike diffusion processes, jump-diffusion processes do not satisfy Aronson-type estimates in general, and hence we cannot apply Gobet's approach.
In this paper, to show the LAN property for jump-diffusion processes, we instead employ the idea of Theorem~1 in Jeganathan~\cite{jeg82}, which uses the $L^2$ regularity condition.
This approach is convenient in the sense that it does not require Aronson-type estimates for transition density functions.
Though the original results in~\cite{jeg82} cannot be applied to triangular array observations, Theorem~2.1 in Fukasawa and Ogihara~\cite{fuk-ogi20} extends this result to triangular array observations, including high-frequency observations of stochastic processes.
Fukasawa and Ogihara~\cite{fuk-ogi20} showed the LAMN property for degenerate diffusion processes by using this result without Aronson-type estimates.
However, since the $L^2$ regularity conditions in~\cite{jeg82, fuk-ogi20} are conditions for expectation, it is difficult to apply them to jump-diffusion processes whose tail is heavier than diffusion processes.
To solve this problem, we weaken the $L^2$ regularity conditions to conditions for conditional expectation that could be applied to heavy-tailed models such as jump-diffusion processes, and show that the LAMN property still holds under this new scheme (Theorem~\ref{L2regu-thm}).

However, there still remains another serious problem to show the LAN property of jump-diffusion processes: the transition density functions of jump-diffusion processes are given by a mixture of different density functions depending on the jump numbers.
In particular, the asymptotic behavior of the density for no jump is quite different to that for the presence of jumps, as indicated in Kohatsu-Higa, Nualart, and Tran~\cite{koh-etal17}.
Kohatsu-Higa, Nualart, and Tran~\cite{koh-etal17} solved this problem by utilizing Malliavin calculous for Wiener-Poisson space and stochastic flows, and obtained the expression of the transition density functions. 
The expression contains the derivative of jump diffusion processes with respect to drift parameters.
However, when the jump structure is additionally parametrized as in our case, not only the jump coefficient but also the associated Poisson random measure may possibly be parametrized in some way.
This makes it difficult to obtain the derivative of jump diffusion processes with respect to jump parameters which will appear in the formal expression of the transition density functions.
For such a reason, it is tough to evaluate the transition density function, which is important to check Theorem~\ref{L2regu-thm}, in the same way as Kohatsu-Higa, Nualart, and Tran~\cite{koh-etal17}.
To deal with this problem, we consider the approximation of transition density functions by thresholding techniques used in Shimizu and Yoshida~\cite{shi-yos06} and Ogihara and Yoshida~\cite{ogi-yos11} in order to construct quasi-maximum-likelihood estimators and Bayes-type estimators.
The thresholding techniques are also used for detecting jumps in processes with jumps (see also \cite{bib-win15, glo-etal18, mai14, man11}) and improving the estimation accuracy of continuous components.
However, it is not clear that the LAN property for the statistical model generated by approximated density functions implies the LAN property for the original model.
We also provide general sufficient conditions for the property (Theorem~\ref{density-approx-thm}).

The rest of this paper is organized as follows. Section~\ref{main-results-section} gives main results. An extended result of Theorem~1 in~\cite{jeg82} is stated in Section~\ref{L2-regu-subsection}. Section~\ref{density-approx-subsection} studies sufficient conditions for the LAMN property using transition density approximation. The LAN property for discrete observations of jump-diffusion
processes is given in Section~\ref{jump-diffusion-LAN-subsection}.
Section~\ref{L2-regu-section} contains the proof of the results in Section~\ref{L2-regu-subsection} following a similar procedure to the proof of Theorem~1 in~\cite{jeg82}.
In Section~\ref{density-approx-section}, we apply the results in Section~\ref{L2-regu-section} to construct a new scheme for the LAMN property via transition density approximation.
In Section~\ref{jump-LAN-section}, we apply the new scheme for the LAMN property via transition density approximation to jump-diffusion processes.

\section{Main results}\label{main-results-section}

\subsection{The LAMN property via a new regularity condition}\label{L2-regu-subsection}

In this subsection, we provide sufficient conditions for the LAMN property {\colorp that is more useful than Theorem~1 in Jeganathan~\cite{jeg82}
and Theorem~2.1 in Fukasawa and Ogihara~\cite{fuk-ogi20} to deal with jump-diffusion processes}.
Some of the assumptions in Theorem~1 of~\cite{jeg82} and Theorem~2.1 of~\cite{fuk-ogi20} are written with respect to expectations.
On the other hand, our new conditions are based on conditional expectations, which is convenient for heavy-tailed noise.

{\colorp Let $\mbbn$ be the set of all positive integers.} Let $\alpha_0\in \Theta$ and $\{P_{\alpha,n}\}_{\alpha\in\Theta}$ be a family of probability measures defined on {\colorp a measurable space} $(\mathcal{X}_n,\mca_n)$ for $n\in\mbbn$,
where $\Theta$ is an open subset of $\mathbb{R}^d$.
We first consider the following slightly weaker condition than the LAMN property.
We denote by $\lVert\cdot \rVert_{{\rm op}}$ the operator norm, {\colorp by $I_l$ the unit matrix of size $l\in\mbbn$,} and by $\top$ the transpose operator for a matrix {\colorp or a vector}.
\begin{description}
\item[Condition (L).] The following two conditions are satisfied for {\colorp $\{P_{\alpha,n}\}_{\alpha\in\Theta,n\in\mbbn}$}.
\begin{enumerate}
\item
There exist a sequence $\{\epsilon_n\}_{n\in\mbbn}$ of nondegenerate matrices,
a sequence $\{V_n(\alpha_0)\}$ of {\colorp $\mca_n$}-measurable $d$-dimensional vectors, and a sequence $\{\mct_n(\alpha_0)\}$
of {\colorp $\mca_n$}-measurable $d\times d$ symmetric matrices such that
$\lVert\epsilon_n\rVert_{{\rm op}}\to 0$ as $n\to\infty$,
\begin{equation}
P_{\alpha_0,n}(\mct_n(\alpha_0) {\rm \ is \ nonnegative \ definite})=1
\end{equation}
for any $n\in\mathbb{N}$, and
\begin{equation}
\log \frac{dP_{\alpha_0+\epsilon_nh,n}}{dP_{\alpha_0,n}}-h^{\top}V_n(\alpha_0)+\frac{1}{2}h^{\top}\mct_n(\alpha_0)h\to 0
\end{equation}
{\colorp as $n\to\infty$} in $P_{\alpha_0,n}$-probability for any $h\in\mathbb{R}^d$.
\item There exists an almost surely {\colorp symmetric,} nonnegative definite {\colorp $d\times d$} random matrix $\mct(\alpha_0)$ such that
\begin{equation*}
\mathcal{L}(V_n(\alpha_0),\mct_n(\alpha_0)|P_{\alpha_0,n})\to \mathcal{L}(\mct^{1/2}(\alpha_0)W,\mct(\alpha_0)),
\end{equation*}
where $W$ is a $d$-dimensional standard normal random variable independent of $\mct(\alpha_0)$.
\end{enumerate}
\end{description}

The following definition of the LAMN property is Definition~1 in~\cite{jeg82}.
\begin{definition}
The sequence of the families {\colorp $\{P_{\alpha,n}\}_{\alpha\in\Theta, n\in\mathbb{N}}$}
satisfies the LAMN condition at $\alpha=\alpha_0\in\Theta$ if Condition (L) is satisfied,
$\epsilon_n$ is a {\colorp symmetric,} positive definite matrix and $P_{\alpha_0,n}(\mct_n(\alpha_0) {\rm \ is \ positive \ definite})=1$
for any $n\in\mathbb{N}$,
and $\mct(\alpha_0)$ is positive definite almost surely.
\end{definition}
We say that {\colorp $\{P_{\alpha,n}\}_{\alpha\in\Theta,n\in\mbbn}$} satisfies LAN if the LAMN condition is satisfied with a nonrandom matrix $\mct(\alpha_0)$.

For proving the LAMN property for diffusion processes by using a localization technique such as Lemma 4.1 in Gobet~\cite{gob01},
Condition~(L) is useful because (L) for the localized model often implies (L) for the original model.
See, for example, the proofs of Theorems~2.4 and 2.5 in~\cite{fuk-ogi20}.

Let $(m_n)_{n=1}^\infty$ be a sequence of positive integers.
{\colorp For any $n$, let} $\{\mathcal{X}_{n,j}\}_{j=1}^{m_n}$ be a sequence of complete, separable metric spaces.
Let $\mathcal{X}_n=\mathcal{X}_{n,1}\times \cdots {\colorp \times} \mathcal{X}_{n,m_n}$ {\colorp and $\mca_n=\mcb(\mcx_n)$, where $\mcb(\mcx_n)$ denotes the Borel $\sigma$-algebra of $\mcx_n$.}
We consider statistical experiments $(\mathcal{X}_n, \mathcal{B}(\mathcal{X}_n), \{P_{\alpha,n}\}_{\alpha\in\Theta})$.
Let $X_j=X_{n,j}:\mathcal{X}_n\to \mathcal{X}_{n,j}$ be the natural projection,
$\bar{X}_j=\bar{X}_{n,j}=(X_1,\cdots, X_j)$,
$\bar{\mathcal{X}}_{n,j}=\mathcal{X}_{n,1}\times \cdots \times \mathcal{X}_{n,j}$, {\colorp $\mca_{0,n}=\{\emptyset,\mcx_n\}$, and $\mca_{j,n}$ is the minimal sub $\sigma$-algebra of $\mca_n$ that ${\bar{X}}_j$ is $\mca_{j,n}$-measurable for $1\leq j\leq m_n$}.
Suppose that there exists a $\sigma$-finite measure $\mu_{n,j}$ on $\mathcal{X}_{n,j}$
such that $P_{\alpha,n}(X_1\in \cdot )\ll \mu_{n,1}$ and
$P_{\alpha,n}(X_j\in \cdot | \bar{X}_{j-1}=\bar{x}_{j-1})\ll \mu_{n,j}$ for $2\leq j\leq m_n$
and $\bar{x}_{j-1}\in \bar{\mcx}_{n,j-1}$.

Let $E_\alpha=E_{\alpha,n}$ denote the expectation with respect to $P_{\alpha,n}$, and let $p_j=p_{j,n}$ be conditional density functions defined by
\begin{equation*}
p_1(\alpha)=\frac{dP_{\alpha,n}(X_1\in \cdot)}{d\mu_{n,1}}, \quad p_j(\alpha)=\frac{dP_{\alpha,n}(X_j\in \cdot \ | \bar{X}_{j-1}=\bar{x}_{j-1})}{d\mu_{n,j}} \quad (2\leq j\leq m_n).
\end{equation*}
Then, we can see that
\begin{equation}\label{condE-eq}
\int p_j(\alpha)g(\bar{x}_{j-1},x_j)d\mu_{n,j}={\colorp E_{\alpha}[g(\bar{X}_{j-1},X_j)|\bar{X}_{j-1}= \bar{x}_{j-1}]}
\end{equation}
{\colorp almost surely} for any bounded Borel function $g:\bar{\mcx}_{n,j}\to \mbbr$.

Next, we describe our assumptions for the LAMN property.
Let $\epsilon_n$ be a $d\times d$ nondegenerate matrix,
and let $\alpha_h=\alpha_0+\epsilon_nh$ for $h\in \mathbb{R}^d$.
\begin{description}
\item[Assumption (A1).] There are random vectors
$\dot{\xi}_{nj}(\alpha_0):{\colorp \bar{\mcx}_{n,j}}\to \mbbr$
such that for every $h\in\mathbb{R}^d$,
\begin{equation}\label{L2-regu-eq}
\sum_{j=1}^{m_n}\int \bigg[\xi_{nj}(\alpha_0,h)-\frac{1}{2}h^{\top} \epsilon_n^\top\dot{\xi}_{nj}(\alpha_0)\bigg]^2d\mu_{n,j}\to 0
\end{equation}
as $n\to\infty$ in $P_{\alpha_0,n}$-probability, where $\xi_{nj}(\alpha_0,h)=\sqrt{p_j(\alpha_h)}-\sqrt{p_j(\alpha_0)}$.
\end{description}

To show the LAMN property, we need to identify the limit distribution of $\log(dP_{\alpha',n}/dP_{\alpha,n})$ under $P_{\alpha,n}$.
This involves the log-likelihood ratio of different probability measures,
which is difficult to deal with for stochastic processes in general.
Gobet~\cite{gob01} dealt with this problem for discretely observed diffusion processes by using estimates from below and above by Gaussian density functions (Aronson estimates) to show the LAMN property.
Condition (A1) also involves transition density functions with different values of the parameter.
However, if $p_j$ is a positive-valued $C^2(\Theta)$ function, an estimate similar to (2.6) in~\cite{fuk-ogi20}, we can replace the left-hand side of (\ref{L2-regu-eq}) with a quantity in which the probability measure of expectation and $p_j$ in the integrand have the same parameter value {\colorp $\alpha_{sh} \ (s\in[0,1])$}, and therefore, we do not need Aronson-type estimates for transition density ratios.
Thus, a scheme with the $L^2$ regularity condition does not require Aronson-type estimates, which is one of the advantage of this scheme.
Furthermore, Condition (A1) is the estimate for conditional expectation unlike (A1) in~\cite{fuk-ogi20}.
Therefore, it is much easier to show (A1) compared with (A1) in~\cite{fuk-ogi20} under the heavy-tailed behavior of jump-diffusion processes.

Define
\begin{equation*}
\eta_{j}{\colorp (\bar{x}_{j-1},x_j)}=\DIV{ll}{
\dot{\xi}_{nj}(\alpha_0)/\sqrt{p_j(\alpha_0)} & {\rm if} \ p_j(\alpha_0)\neq 0 \\
0 & {\rm otherwise} }
\end{equation*}
{\colorp We use abbreviation $\eta_j$ both for the random variable $\eta_j(\bar{X}_{j-1},X_j)$ and for the function $\eta_j(\bar{x}_{j-1},x_j)$ when  there is no confusion.
The same is true for other functions of $(\bar{x}_{j-1},x_j)$.}

\begin{equation}\label{TnWn-def}
\mct_n= \epsilon_n^\top\sum_{j=1}^{m_n}E_{\alpha_0}[\eta_{j}\eta_{j}^{\top}|\mca_{j-1,n}]\epsilon_n, \quad
{\rm and} \quad V_n= \epsilon_n^\top\sum_{j=1}^{m_n}\eta_j.
\end{equation}
\begin{description}
\item[Assumption (A2).] {\colorp There exists $n_0\in\mbbn$ such that} $E_{\alpha_0}[|\eta_j|^2|\mca_{j-1,n}]<\infty$ and $E_{\alpha_0}[\eta_{j}|\mca_{j-1,n}]=0$, $P_{\alpha_0,n}$-almost surely for every $1\leq j\leq m_n$ {\colorp and $n\geq n_0$.} 
\item[Assumption (A3).] For every $\epsilon >0$ and $h\in \mathbb{R}^d$,
\begin{equation*}
\sum_{j=1}^{m_n}E_{\alpha_0}[|h^{\top}\epsilon_n^\top\eta_{j}|^21_{\{|h^{\top}\epsilon_n^\top\eta_{j}|>\epsilon\}}|\mca_{j-1,n}]\to 0
\end{equation*}
{\colorp as $n\to\infty$} in $P_{\alpha_0,n}$-probability.
\item[Assumption (A4).]
There exists a random $d\times d$ symmetric matrix $\mct$ such that
$P(\mct \ {\rm is \ nonnegative \ definite})=1$ and
\begin{equation*}
\mathcal{L}((V_n, \mct_n)|P_{\alpha_0,n}) \to \mathcal{L}(\mct^{1/2}W,\mct),
\end{equation*}
where $W \sim N(0,I_d)$ independent of $\mct$.
\end{description}
Conditions (A2)--(A4) are similar to Conditions (A2)--(A5) in~\cite{fuk-ogi20}.
However, (A3) and (A4) in~\cite{fuk-ogi20} are replaced by estimates for conditional expectations and a tightness property that is trivially satisfied under (A4).

\begin{theorem}\label{L2regu-thm}
Assume $(A1)$--$(A4)$. Then the family $\{P_{\alpha,n}\}_{\alpha,n}$
satisfies Condition (L) with $\mct_n$ and $V_n$ in (\ref{TnWn-def}).
If further $\mct$ in (A4) is positive definite almost surely
and $\epsilon_n$ is a {\colorp symmetric,} positive definite matrix for any $n\in\mbbn$,
then $\{P_{\alpha,n}\}_{\alpha,n}$ satisfies the LAMN property at $\theta=\theta_0$.
\end{theorem}

\begin{remark}\label{LAMN-remark}
As in Remark 2.1 of~\cite{fuk-ogi20}, if Condition (L) is satisfied, $\epsilon_n$ is {\colorp symmetric and} positive definite for any $n\in\mbbn$, and $\mct$ is positive definite almost surely, then we can easily show the LAMN property by replacing $\mct_n$ with
\EQQ{\dot{\mct}_n=\mct_n1_{\{\mct_n {\rm \ is \ positive \ definite}\}}
+I_d1_{\{\mct_n {\rm \ is \ not \ positive \ definite}\}}.}
\end{remark}

\subsection{The LAMN property via transition density approximation}\label{density-approx-subsection}

Theorem~\ref{L2regu-thm} is an important tool when we show the LAN property for discrete observations of jump-diffusion processes because
this result requires neither Aronson-type estimates nor an expectation-type $L^2$ regularity condition. 
The another important issue to show the LAN property is to handle the mixture of density functions that behave quite differently depending on the jump numbers.
We use the thresholding techniques developed in Shimizu and Yoshida~\cite{shi-yos06} and Ogihara and Yoshida~\cite{ogi-yos11} to deal with this issue.
We approximate the transition density functions of jump-diffusion processes with thresholding density functions whose asymptotic behaviors are much easier to deal with.
We will show that the LAN property of the original model is proved under some conditions on the approximating density functions.

Let $\tilde{p}_1(\alpha)=\tilde{p}_1(\alpha,x_1)$ and $\tilde{p}_j(\alpha)=\tilde{p}_j(\alpha,x_j,\bar{x}_{j-1})$ be nonnegative-valued functions such that $\tilde{p}_1(\alpha,\cdot)$ is measurable and the mapping $(\bar{x}_{j-1},A)\mapsto \int_A \tilde{p}_j(\alpha,x_j,\bar{x}_{j-1})\mu_{n,j}(dx_j)$ is a transition kernel for $2\leq j\leq m_n$.
We emphasize that $\tilde{p}_1(\alpha,\cdot)$ and $\tilde{p}_j(\alpha,\cdot,\bar{x}_{j-1})$ are not supposed to be probability measures.
This is important in the sense that we can consider normalized probability measures on sets that do not contain original rare events.
We introduce associated normalizing constants $d_1(\alpha)=\int\tilde{p}_1(\alpha)dx_1$ and $d_j(\bar{x}_{j-1},\alpha)=\int \tilde{p}_j(\alpha)dx_j$ for $2\leq j \leq m_n$.
Assume $d_j(\bar{x}_{j-1},\alpha)$ is nonzero and finite for any $(\bar{x}_{j-1},\alpha)$,
and let $\tilde{P}_{\alpha,n}$ be a probability measure defined by $\tilde{P}_{\alpha,n}=\prod_{j=1}^{m_n}(\tilde{p}_j(\alpha)/d_j(\alpha))({\colorp \bigotimes_{j=1}^{m_n}}\mu_{n,j}(dx_j))$,
where $\bar{x}_0=\emptyset$.

Let ${\colorp K_{n,j}\in \mca_{j,n}}$ for $1\leq j\leq m_n-1$. Let $D_{j,h}(\bar{x}_{j-1},t)=d_j(\bar{x}_{j-1},\alpha_{th})$ for $t\in [0,1]$ and $h\in \mbbr^d$ such that $(\alpha_{th})_{t\in [0,1]}\subset \Theta$.
Let
\EQQ{\zeta_{j,t}^{l,h}=\frac{(\frac{d}{dt})^l\tilde{p}_j(\alpha_{th})}{\tilde{p}_j(\alpha_{th})}1_{\{\tilde{p}_j(\alpha_{th})\neq 0\}}}
for $l\in \mbbn$ and $h\in \mbbr^d$.

\begin{description}
\item[Assumption (B1).] For any $\epsilon>0$ there exists $N\in\mathbb{N}$ such that
\begin{equation}\label{tail-prob-est1}
\sup_{\alpha\in \Theta}P_{\alpha,n}(\cup_{j=1}^{ m_n-1}K_{n,j}^c)<\epsilon
\end{equation}
and
\begin{equation}\label{tail-prob-est2}
\sup_{\alpha\in \Theta}\tilde{P}_{\alpha,n}(\cup_{j=1}^{ m_n-1}K_{n,j}^c)<\epsilon
\end{equation}
for $n\geq N$. Moreover,
\begin{equation}\label{p-diff-est}
m_n\max_{1\leq j\leq m_n}\sup_{\alpha\in \Theta,\bar{x}_{j-1}\in {\colorp \bar{X}_{j-1}(K_{n,j-1})}}\int |p_j(\alpha)-\tilde{p}_j(\alpha)|\mu_{n,j}(dx_j)\to 0
\end{equation}
as $n\to\infty$.
\end{description}

{\colorp Here and in the following, we ignore $\bar{x}_0\in \bar{X}_0(K_{n,0})$ in the range of the supremum.}
(B1) implies that $\tilde{P}_{\alpha,n}$ approximates $P_{\alpha,n}$ well except on a rare event.
A typical example of a rare event is that $x_1, \dots, x_n, \dots$ have large magnitude.
On such an event, it is often difficult to evaluate a difference of mass measured by $\tilde{P}_{\alpha,n}$ and $P_{\alpha,n}$.
As for an application to jump-diffusion processes, we set
\begin{equation*}
K_{n,j}=\left\{{\colorp (x_l)_{l=0}^n\subset \mbbr^{m(n+1)}}\middle| \max_{0\leq l\leq j} |x_l|\leq n^\del\right\}
\end{equation*}
for small enough $\del>0$, and this makes it possible to obtain \eqref{p-diff-est}.
For more details, see Section \ref{B1B2-subsection}.

Since the approximation probability measure $\tilde{P}_{\alpha,n}$ contains normalizing constants $d_1,\dots,d_{m_n}$, to check \eqref{tail-prob-est2} may seem cumbersome.
Then the following lemma is helpful.

\begin{lemma}\label{B1-suff-lemma}
Assume (\ref{p-diff-est}), that $\tilde{p}_j(\alpha,x_j,\bar{x}_{j-1})\leq p_j(\alpha,x_j,\bar{x}_{j-1})$ $\mu_{n,1}\otimes \cdots \otimes \mu_{n,j}$-almost everywhere in $\bar{x}_j$ for any $\alpha$, 
and that for any $\epsilon>0$, there exists $N\in\mathbb{N}$ such that (\ref{tail-prob-est1}) for $n\geq N$.
Then, for any $\epsilon>0$, there exists $N'\in\mathbb{N}$ such that (\ref{tail-prob-est2}) for $n\geq N'$.
\end{lemma}

The following theorem ensures that the LAMN property of $(\tilde{P}_{\alpha,n})_{\alpha,n}$ implies the LAMN property of $(P_{\alpha,n})_{\alpha,n}$ under (B1).

\begin{theorem}\label{log-likelihood-approx-thm}
Assume (B1). Then, $\sup_{\alpha\in\Theta}\lVert P_{\alpha,n}-\tilde{P}_{\alpha,n}\rVert\to 0$ {\colorp as $n\to\infty$. If further, for any $\epsilon>0$ and $h\in \mbbr^d$, there exists $\delta>0$ such that 
\EQ{\label{tildeP-cond} \limsup_{n\to\infty}\tilde{P}_{\alpha_0,n}\bigg(\frac{d\tilde{P}_{\alpha_h,n}}{d\tilde{P}_{\alpha_0,n}}<\delta\bigg)<\epsilon,}
then
}
\begin{equation*}
\log\frac{dP_{\alpha_h,n}}{dP_{\alpha_0,n}}-\log\frac{d\tilde{P}_{\alpha_h,n}}{d\tilde{P}_{\alpha_0,n}}\to 0
\end{equation*}
as $n\to\infty$ in $P_{\alpha_0,n}$- and $\tilde{P}_{\alpha_0,n}$-probability for any $h\in\mathbb{R}^d$.
\end{theorem}

{\colorp For a vector $x=(x_1,\cdots, x_k)$, we denote $\partial_x^l=(\frac{\partial^l}{\partial x_{i_1}\cdots \partial x_{i_l}})_{i_1,\cdots,i_l=1}^k$.}

\begin{description}
\item[Assumption (B2).]
{\colorp For any $h\in\mbbr^d$, there} exists $N\in \mbbn$ such that
$(\frac{d}{dt})^lp_j(\alpha_{th})$, $(\frac{d}{dt})^l\tilde{p}_j(\alpha_{th})$ and $\partial_t^lD_{j,h}$ exist and is continuous for $n\geq N$, $t\in [0,1]$,
and $l\in\{0,1,2\}$,
almost everywhere in $\bar{x}_{j-1}\in {\colorp \bar{X}_{j-1}(K_{n,j-1})}$.
Moreover, {\colorp there exists $\delta>0$ such that}
\begin{equation}\label{zeta-integrable}
\sup_{|t|\leq \delta, \bar{x}_{j-1}\in {\colorp \bar{X}_{j-1}(K_{n,j-1})}}\int|\zeta_{j,t}^{l,h}|\tilde{p}_j(\alpha_{th})\mu_{n,j}(dx_j)<\infty
\end{equation}
and
\begin{equation}\label{D-del-est}
m_n^{1/l}\max_{1\leq j\leq m_n}\sup_{t\in [0,1],\bar{x}_{j-1}\in {\colorp \bar{X}_{j-1}(K_{n,j-1})}}|\partial_t^lD_{j,h}(\bar{x}_{j-1},t)|\to 0 \quad {\rm as} \quad {\colorp n}\to \infty
\end{equation}
for {\colorp $l\in \{1,2\}$} and $n\geq N$.
\end{description}

Lebesgue's dominated convergence theorem yields (\ref{D-del-est}) if
\EQ{\label{B2-suff} m_n^{1/l}\max_{1\leq j\leq m_n}\sup_{t\in [0,1],\bar{x}_{j-1}\in {\colorp \bar{X}_{j-1}(K_{n,j-1})}}
\int |\partial^l_\alpha p_j(\alpha_{th})-\partial^l_\alpha \tilde{p}_j(\alpha_{th})|\mu_{n,j}(dx_j)\to 0
}
for $h\in \mbbr^d$, $l\in \{1,2,3\}$, and $n\geq N$ (see (\ref{tildePderivaEst}) and {\colorp (\ref{eta-j-zero})} for the details).
For the setting in Section~\ref{jump-diffusion-LAN-subsection}, it is not easy to check (\ref{B2-suff}) because of the heavy-tailed behavior. So we directly check (\ref{D-del-est}) in Section~\ref{jump-LAN-section}.

Let $(e_i)_{i=1}^d$ be the standard unit vectors in $\mbbr^d$, and let ${\colorp \tilde{\eta}_j(\bar{x}_{j-1},x_j)=(\zeta_{j,0}^{1,e_1},\cdots, \zeta_{j,0}^{1,e_d})1_{\bar{X}_{j-1}(K_{n,j-1})}(\bar{x}_{j-1})}$.
Let $\tilde{E}_{\alpha}$ denote the expectation with respect to $\tilde{P}_{\alpha,n}$.
Let
\EQ{\label{tildeTtildeV-def} \tilde{\mct}_n=\sum_{j=1}^{m_n}\tilde{E}_{\alpha_0}[\tilde{\eta}_{j}\tilde{\eta}_{j}^{\top}|\mca_{j-1,n}]
\quad {\rm and} \quad \tilde{V}_n=\sum_{j=1}^{m_n}\tilde{\eta}_j.}
We further assume the following conditions.
\begin{description}
\item[Assumption (B3).]
$\tilde{E}_{\alpha_0}[|\zeta_{j,t}^{1,h}|^2|\mca_{j-1,n}]<\infty$ and the zero points of $\tilde{p}_j$ do not depend on $\alpha\in \Theta$ $\tilde{P}_{\alpha_0,n}$-almost surely for $1\leq j\leq m_n$, and
\begin{equation*}
\sum_{j=1}^{m_n}\sup_{t\in [0,1]}\tilde{E}_{\alpha_{th}}[|\zeta_{j,t}^{2,h}|^2
+|\zeta_{j,t}^{1,h}|^4|\mca_{j-1,n}]\to 0
\end{equation*}
as $n\to\infty$ in $\tilde{P}_{\alpha_0,n}$-probability.
\item[Assumption (B4).]
There exists a random $d\times d$ symmetric matrix $\mct$ such that $P[\mct \ {\rm is \ nonnegative \ definite}]=1$ and
\begin{equation*}
\mathcal{L}((\tilde{V}_n, \tilde{\mct}_n)|\tilde{P}_{\alpha_0,n}) \to \mathcal{L}(\mct^{1/2}W,\mct),
\end{equation*}
where $W \sim N(0,I_d)$ {\colorp independent of $\mct$}.
\end{description}

Conditions (B3) and (B4) are conditions for the asymptotic behavior of functions of $\partial_t\tilde{p}_j$ (not $\partial_t p_j$).
This fact is important when we discuss the LAN property of jump-diffusion procesess in the following section. While the asymptotic behavior of the transition density functions of jump-diffusion processes is difficult to deal with, that of thresholding density functions is much easier to handle.
The next theorem ensures that we need to consider only the latter when we show the LAMN property of the original model.

\begin{theorem}\label{density-approx-thm}
Assume (B1)--(B4). Then, the family $\{P_{\alpha,n}\}_{\alpha,n}$ of probability measures satisfies Condition (L) with $\tilde{\mct}_n$ and $\tilde{V}_n$ in (\ref{tildeTtildeV-def}).
If further $\mct$ in (B4) is positive definite almost surely and $\epsilon_n$ is {\colorp symmetric and} positive definite for any $n\in\mbbn$,
then $\{P_{\alpha,n}\}_{\alpha,n}$ satisfies the LAMN property at $\alpha=\alpha_0$.
\end{theorem}

If $\mct$ is nonrandom (which corresponds to the case of {LAN}), we can simplify Condition (B4).
\begin{description}
\item[Assumption (B4$'$).] There exists a nonrandom $d\times d$ symmetric, nonnegative definite matrix $\mct$ such that
$\tilde{\mct}_n\to \mct$ in $\tilde{P}_{\alpha_0,n}$-probability.
\end{description}

\begin{corollary}\label{B4'-cor}
Assume (B1), (B2), (B3), and (B4$'$). Then $\{P_{\alpha,n}\}_{\alpha,n}$ satisfies Condition (L). If further $\mct$ in (B4$'$) is positive definite and $\epsilon_n$ is {\colorp symmetric and} positive definite for any $n\in\mbbn$, $\{P_{\alpha,n}\}_{\alpha,n}$ satisfies the LAN property at $\alpha=\alpha_0$.
\end{corollary}

\begin{remark}
Even when $\mct$ is random, Sweeting~\cite{swe80} is useful to omit checking the convergence of $\tilde{V}_n$ in (B4).
\end{remark}

\begin{remark}
We expect that such techniques of transition density approximation can be applied to models other than the jump-diffusion model.
If we can find an approximation of the transition density function of a statistical model such that the asymptotic behavior of the approximation can be specified, these techniques enable us to show the LAMN property of the statistical model.
One such an example is the statistical model of nonsynchronously observed diffusion processes in Ogihara~\cite{ogi15}. The likelihood function is given by the integral of the likelihood function for synchronized observations with respect to unobserved variables.
The LAMN property for this model is shown by introducing the likelihood approximation obtained by cutting off the domain of integration, and identifying the asymptotic behavior of the approximated likelihood function (see Lemma 4.3 and subsequent discussions in~\cite{ogi15}).
So the techniques in this section enable us to simplify the proof of LAMN for this model.
\end{remark}

\subsection{The LAN property for jump-diffusion processes}\label{jump-diffusion-LAN-subsection}

In this section, we show the LAN property of jump-diffusion processes.
{\colorp Let $(\Omega, \mathcal{F},(\mcf_t)_{t\geq 0}, P)$ be a stochastic basis.
Let $\Theta_i\subset \mbbr^{d_i}$ be an open set for $i\in\{1,2\}$, and $\Theta=\Theta_1\times \Theta_2$.
We set $\alpha=(\sigma,\theta)\in \Theta_1\times \Theta_2$ and its true value is denoted by $\alpha_0=(\sigma_0,\theta_0)$.
For any $\alpha\in \Theta$, let} $X^\alpha=(X_t^\alpha)_{t\geq 0}$ be an $m$-dimensional c\`adl\`ag ${\bf F}$-adapted process satisfying a stochastic differential equation:
\EQ{\label{jump-SDE} dX_t^\alpha=a(X_{t}^\alpha,\theta)dt+b(X_{t}^\alpha,\sigma)dW_t+\int_E zN_\theta(dt,dz),}
where {\colorp ${\bf F}=(\mcf_t)_{t\geq 0}$, $E=\mbbr^m\backslash\{0\}$, $W=(W_t)_{t\geq 0}$ is an $m$-dimensional standard ${\bf F}$-Wiener process, and
$N_\theta$ is a Poisson random measure on $\mbbr_+\times E$ relative to ${\bf F}$, whose mean measure is $f_\theta(z)dzdt$ with $\int_E f_\theta(z)dz<\infty$.}
The coefficients $a: \mbbr^m\times \Theta_2\mapsto\mbbr^m$ and $b: \mbbr^m\times \Theta_1\mapsto {\colorp \mbbr^m}\otimes \mbbr^m$ are measurable functions and satisfy Assumption (C1) below.
We assume that the distribution of $X_0^\alpha$ does not depend on $\alpha\in \Theta$.
We denote $X_t=X_t^{\alpha_0}$ and $N(dt,dz)=N_{\theta_0}(dt,dz)$.
We suppose that we observe high-frequency data $\{X_{kh_n}\}_{k=0}^n$ from the solution process $X=(X_t)_{t\geq 0}$. $\{h_n\}_{n\in\mbbn}$ is a positive sequence with $h_n\to 0$ and $nh_n\to \infty$.
For matrices $(M_i)_{i=1}^l$, let
\EQQ{{\rm diag}((M_i)_{i=1}^l)=\MAT{ccc}{M_1 & O & O \\ O & \ddots & O \\ O & O & M_l}.}

\begin{description}
\item[Assumption (C1).]
The derivatives $\partial_x^i\partial_\theta^j a(x,\theta)$ and $\partial_x^i\partial_\sigma^jb(x,\sigma)$ exist and are continuous on {\colorp $\mbbr^m\times \Theta_2$ and $\mbbr^m\times \Theta_1$}, respectively, for $i,j\in \{0,1,2,3,4\}$ such that $i+j\leq 4$.
Moreover, there exist positive constants $C_1$ and $\kappa$ such that
\begin{equation*}
|a(x,\theta)|\leq C_1(1+|x|), \quad |\partial_x a(x,\theta)|+|b(x,\sigma)|+|\partial_xb(x,\sigma)|\leq C_1,
\end{equation*}
\begin{equation*}
|\partial_x^i\partial_\theta^j a(x,\theta)|+|\partial_x^i\partial_\sigma^jb(x,\sigma)|\leq C_1(1+|x|)^\kappa
\end{equation*}
for all $i,j\in \{0,1,2,3,4\}$ satisfying $i+j\leq 4$, and $(\theta,\theta',\sigma,\sigma',x)$.

\item[Assumption (C2).] $b(x,\sigma)$ is symmetric, positive definite, and there exists a positive constant $C_2$ such that
\begin{equation*}
C_2^{-1}I_m\leq b(x,\sigma)\leq C_2I_m
\end{equation*}
for any $x$ and $\sigma$.

\item[Assumption (C3).] $X$ is ergodic; that is, there exists a stationary distribution {\colorp $\pi$} such that
\EQ{\label{ergod-cond} \frac{1}{T}\int_0^Tg(X_t)dt\overset{P}\to\int g(x)d\pi{\colorp (x)}}
as $T\to\infty$ for any $\pi$-integrable function $g$. 
Moreover,
\EQ{\label{moment-cond} \sup_{\alpha\in \Theta,t\geq 0}E[|X_t^\alpha|^q]<\infty}
for $q>0$.
\end{description}

Let $F_\theta$ be a density function satisfying $f_\theta= \lambda F_\theta$ with some positive constant $\lambda=\lambda(\theta)$.
Hereafter we write the support of any function $g$ and its boundary {\colorp in $E$} as $\supp(g)$ and $\p \supp(g)$, respectively.
{\colorp Let $d(z,A)=\inf_{y\in A}|z-y|$ for $z\in \mbbr^m$ and $A\subset \mbbr^m$ ($d(z,\emptyset)=\infty$ by convention).}

\begin{description}
\item[Assumption (C4).]
\begin{enumerate}
\item The zero points of $F_\theta$ do not depend on $\theta$.
\item {\colorp The derivative $\partial_\theta^l\lambda$ exists and bounded for $l\in \{0,1,2,3\}$.}
\item There exist constants $\ep>0$, $\rho \in (0,1/2)$, and $N_0\in\mbbn$ fulfilling that 
\EQ{\label{jzero-cond}\int_{\{z: d(z, \partial\supp(F_\theta))\leq h_n^\rho\}} F_\theta(z)dz \leq h_n^\ep}
for all $n\geq N_0$.
\item {\colorp The derivative $\partial_\theta^l f_\theta(z)$ exists and is continuous with respect to $\theta\in \Theta_2$ for any $l\in \{0,1,2,3\}$ and $z\in E$. Moreover,} there exist constants $\gamma\geq 0$, {\colorp $C_3>0$,} and ${\colorp \epsilon'}>0$ such that
\EQ{\label{jcond} |F_\theta(z)|1_{\{|z|\leq {\colorp \epsilon'}\}}\leq C_3|z|^\gamma,
 \quad |\partial_\theta^l\log f_\theta(z)|1_{\{F_\theta(z)\neq0\}}\leq C_3(1+|z|)^{C_3},}
\EQ{\label{jcond2} |\partial_\theta^l\log f_\theta(z_1)-\partial_\theta^l\log f_\theta(z_2)|1_{\{F_\theta(z_1)F_\theta(z_2)\neq0\}}\leq C_3|z_1-z_2|(1+|z_1|+|z_2|)^{C_3}}
for any {\colorp $z,z_1,z_2\in E$, $\theta\in \Theta_2$, and $l\in \{1,2,3\}$}.
\item $\sup_\theta \int|z|^pf_\theta(z)dz<\infty$ for any $p\geq 1$, and {\colorp there exists $\eta>0$ such that}
\EQ{\label{balance-cond} n^{1+\eta}h_n^{1+((m+\gamma)/2)\wedge 1}\to 0}
as $n\to\infty$.
\end{enumerate}
\end{description}

Let $\Gamma={\rm diag}((\Gamma_1, \Gamma_2))$,
where $S(x,\sigma)=b^2(x,\sigma)$,
\EQNN{[\Gamma_1]_{ij}&=\frac{1}{2}\int {\rm tr}(\partial_{\sigma_i}SS^{-1}\partial_{\sigma_j}SS^{-1})(x,\sigma_0)d\pi(x), \\
[\Gamma_2]_{ij}&=\int (\partial_{\theta_i} a)^\top S^{-1}(\partial_{\theta_j} a)(x,\alpha_0)d\pi(x)
+{\colorp \int_E} \frac{\partial_{\theta_i} f_{\theta_0}\partial_{\theta_j} f_{\theta_0}}{f_{\theta_0}} 1_{\{f_{\theta_0}\neq 0\}}(y)dy.
}

\begin{description}

\item[Assumption (C5).] $\Gamma$ is positive definite.
\end{description}

Regarding our technical assumptions, we make some comments below.
\begin{itemize}
\item Under (C1), the existence and uniqueness of the solution is ensured (for details, see Applebaum~\cite{app09}).
(C1) is also important in considering the derivatives of the flow and Malliavin calculus on the continuous part of \eqref{jump-SDE}.
\item
For sufficient conditions of ergodicity (\ref{ergod-cond}), we refer the reader to Masuda~\cite{mas08}.
We need the moment condition (\ref{moment-cond}) of $X_t^\alpha$ uniformly in $\alpha$.
This condition is a bit stronger than the one usually assumed in the studies of statistical estimation for jump-diffusion processes (estimate for only $\alpha=\alpha_0$) since evaluation of transition densities around $\al_0$ is essential for the LAN property.
However, this condition can be shown similarly to a standard procedure.
See Theorem~2.2 in Masuda~\cite{mas07} for the details.
\item Since we cannot observe fluctuation of jumps directly, we replace it by the increments of $X$ exceeding a threshold on the estimation of $\theta$. However, these increments may not belong to the support of $F_\theta$ typically for one-sided jumps or bounded jumps. 
\eqref{jzero-cond} is useful for controlling such a ({\colorp $\mcf_{t_{j-1}}$}-conditional) probability, for more details, see Lemma \ref{incprob}. 
\item {\colorp Suppose that $\sup_\theta \int_E F_\theta^p(z)dz<\infty$ for some $p>1$ and $\partial {\rm supp}(F_\theta)=\{z_1,\cdots, z_k\}$ for some $k\in \mbbn$, and $z_j\in E \ (1\leq j\leq k)$. 
Then the set $\{z;d(z,\partial {\rm supp}(F_\theta))\leq h_n^\rho\}$ is included in union of 	closed $k$ balls centered at $z_1,\cdots, z_k$ of radius $h_n^\rho$, and hence H\"older's inequality yields
\EQQ{\int_{\{z;d(z,\partial {\rm supp}(F_\theta))\leq h_n^\rho\}}F_\theta(z)dz
\leq \bigg(\int_E F_\theta^p(z)dz\bigg)^{1/p}\bigg(\int_{\{z;d(z,\partial {\rm supp}(F_\theta))\leq h_n^\rho\}}dz\bigg)^{1/q}\leq Ck^{1/q}h_n^{\rho m/q},}
where $q=p/(p-1)$.
Then (\ref{jzero-cond}) is satisfied for sufficiently large $n$. 
}
\item When $m\geq 2$, (\ref{balance-cond}) becomes $n^{1+\eta}h_n^2\to 0$ for some $\eta>0$, which is almost the same as the one usually required in the study of statistical estimation for jump-diffusion processes.
We can say the same thing when $m=1$ and $\gamma\geq 1$. This condition is weaker than the corresponding condition in Shimizu and Yoshida~\cite{shi-yos06} ($\gamma>3$ is required).
We can also consider the case $m=1$ and $\gamma\in [0,1)$. In this case, the convergence rate of $h_n$ becomes restrictive ($n^{2+\eta}h_n^3\to 0$ for some $\eta>0$ in the worst case).
These things happen because {\colorp we need to detect jumps by using the increment $|X_{kh_n}-X_{(k-1)h_n}|$. Roughly speaking, for $\rho\in (0,1/2)$, we have $|X_{kh_n}-X_{(k-1)h_n}|\leq h_n^\rho$ with high probability if no jumps in $((k-1)h_n,kh_n]$. Then we judge jumps occur when $|X_{kh_n}-X_{(k-1)h_n}|>h_n^\rho$.}
If the dimension of $X_t$ is large or $\gamma$ is large, then the probability that the absolute jump size is equal to or less than $h_n^\rho$ becomes very small, and consequently, jump detection and approximation by thresholding densities work well.
{\colorp Otherwise, we need to set $h_n$ small to detect jumps.}
\end{itemize}


{\colorp
\begin{example}
Condition (C4) is a bit complicated, so we will see some examples of $F_\theta$ which satisfies (C4). Let $\lambda$ is a smooth function of $\theta$ satisfying
$\sup_\theta|\partial_\theta^l\lambda|<\infty$ 
for $l\in \{0,1,2,3\}$ and $\inf_\theta \lambda>0$.
\begin{enumerate}
\item Let
\EQQ{F_\theta(z)=\frac{1}{(2\pi \det \Sigma)^{m/2}}\exp\bigg(-\frac{1}{2}(z-\mu)^\top \Sigma^{-1}(z-\mu)\bigg), \quad ({\rm normal \  distribution})}
where $\mu$ and $\Sigma$ are smooth $\mbbr^m$- and $\mbbr^m\otimes \mbbr^m$-valued functions of $\theta$, respectively, such that
$\sup_\theta(|\PT^l\mu|\vee \lVert \PT^l\Sigma\rVert_{{\rm op}})<\infty$
for $l\in \{0,1,2,3\}$ and $\sup_\theta \lVert \Sigma^{-1}\rVert_{{\rm op}}<\infty$.
Then we can easily check (\ref{jzero-cond})--(\ref{jcond2}). Therefore, (C4) is satisfied if there exists $\eta>0$ such that
\EQ{\label{hn-cond} \DIV{ll}{n^{2+\eta}h_n^3  &\to 0 \quad (m=1) \\
n^{1+\eta}h_n^2&\to 0 \quad (m\geq 2)}}
\item Let $m=1$ and
\EQQ{F_\theta(z)=\frac{1}{\Gamma(\alpha)\beta^\alpha}z^{\alpha-1}e^{-z/\beta}1_{\{z>0\}}, \quad ({\rm Gamma \ distribution})}
where $\alpha$ and $\beta$ are smooth $\mbbr$-valued functions of $\theta$ such that $\sup_\theta(|\partial_\theta^l\alpha|\vee |\partial_\theta^l\beta|)<\infty$ for $l\in \{0,1,2,3\}$, $\inf_\theta \beta>0$, and $\inf_\theta \alpha\geq 1$.
Then $\partial {\rm supp}(F_\theta)=\emptyset$, that implies (\ref{jzero-cond}). Moreover, we have
\EQQ{\log f_\theta(z)1_{\{z>0\}}=\bigg\{(\alpha-1)\log z-\frac{z}{\beta}-\alpha\log \beta-\log \Gamma(\alpha)+\log \lambda\bigg\} 1_{\{z>0\}}.}

If $\partial_\theta \alpha\equiv 0$ (for example, the case of exponential distributions $\alpha\equiv 1$), then $\log f_\theta$ satisfies (\ref{jcond}) and (\ref{jcond2}),
and (C4) holds if there exists $\eta>0$ such that $n^{1+\eta}h_n^{1+(\alpha/2)\wedge 1}\to 0$ as $n\to\infty$.

If $\partial_\theta \alpha\neq 0$ for some $\theta'$, then (\ref{jcond}) is not satisfied because $\lim_{z\searrow 0}|\PT \log f_{\theta'}(z)| \to \infty$, and hence (C4) does not hold. 

\item Let $m=1$ and 
\EQQ{F_\theta(z)=\frac{1}{\Gamma(\alpha_1)\beta_1^{\alpha_1}}z^{\alpha_1-1}e^{-z/\beta_1}1_{\{z>0\}}+\frac{1}{\Gamma(\alpha_2)\beta_2^{\alpha_2}}(-z)^{\alpha_2-1}e^{z/\beta_2}1_{\{z<0\}}, \quad ({\rm two\mathchar`-sided~Gamma~distribution})}
where $\alpha_1$, $\alpha_2$, $\beta_1$, and $\beta_2$ are smooth $\mbbr$-valued functions of $\theta$ such that $\sup_\theta(|\partial_\theta^l\alpha_j|\vee |\partial_\theta^l\beta_j|)<\infty$, $\inf_\theta \beta_j>0$, and $\inf_\theta \alpha_j\geq 1$ for $l\in \{0,1,2,3\}$ and $j\in\{1,2\}$.

Similarly to the above example, (C4) holds if $\partial_\theta \alpha_j\equiv 0$ for $j\in \{1,2\}$ and there exists $\eta>0$ such that $n^{1+\eta}h_n^{1+(\alpha_1\wedge \alpha_2\wedge 2)/2}\to 0$. (C4) does not satisfied if $\partial_\theta \alpha_j\neq 0$ for some $j\in\{1,2\}$ and $\theta$.
\end{enumerate}
\end{example}
}

$ $

Let $\{P_{\alpha,n}\}_{\alpha,n}$ be the family of probability measures generated by $(X_{kh_n}^\alpha)_{k=0}^n$.
\begin{theorem}\label{jump-LAN-thm}
Assume (C1)--(C5). Then, $\{P_{\alpha,n}\}_{\alpha,n}$ satisfies LAN at $\alpha=\alpha_0$
with $\mct(\alpha_0)=\Gamma$ and $\epsilon_n={\rm diag}(n^{-1/2}I_{d_1},(nh_n)^{-1/2}I_{d_2}))$.
\end{theorem}

\begin{remark}
Theorem~9.1 in Chapter II of Ibragimov and Has'minski{\u\i}~\cite{ibr-has81} yields the convolution theorem for this model.
We can see that $\Gamma^{-1}$ coincides with the asymptotic variances of the quasi-maximum-likelihood estimator $\hat{\alpha}_n=(\hat{\sigma}_n,\hat{\theta}_n)$ and the Bayes-type estimator $\tilde{\alpha}_n=(\tilde{\sigma}_n,\tilde{\theta}_n)$ in Shimizu and Yoshida~\cite{shi-yos06}
and Ogihara and Yoshida~\cite{ogi-yos11}, respectively. Then we can show that these estimators are asymptotically efficient in the sense of the convolution theorem.
\end{remark}

\begin{remark}
Because $\hat{\alpha}_n$ is asymptotically efficient and the asymptotic covariance of $\hat{\sigma}_n$ and $\hat{\theta}_n$ is equal to zero, Theorem~\ref{jump-LAN-thm}
allows us to construct Wald-type tests testing $H: \sigma=\sigma_0$ (resp.\ $\theta=\theta_0$) against $K: \sigma\neq \sigma_0$ (resp.\ $\theta\neq\theta_0)$.
These tests are asymptotically {\colorp uniformly most powerful} in the sense of Sections 4 and 5 in Choi, Hall, and Schick~\cite{cho-etal96} (see Section~7 and Theorems~2 and 3 in~\cite{cho-etal96})
(though the scaling matrix $\epsilon_n$ is assumed to be $I_d/\sqrt{n}$ in~\cite{cho-etal96}, their proofs remain valid for our setting).
\end{remark}

\begin{remark}
We can generalize Theorem~\ref{jump-LAN-thm} when the jump part in (\ref{jump-SDE}) is given by $\int_E c(X_{t-}^\alpha,z,\theta){\colorp N_\theta}(dt,dz)$ under
similar conditions to [H6], [H7], and [G1] in Ogihara and Yoshida~\cite{ogi-yos11} by introducing the function $\Psi_\theta(y,x)$ in Shimizu and Yoshida~\cite{shi-yos06} and Ogihara and Yoshida~\cite{ogi-yos11}. However, we adopt $c(x,z,\theta)=z$ in our setting to avoid excessive complexity.
\end{remark}

\section{The LAMN property via a conditional $L^2$ regularity condition}\label{L2-regu-section}

In this section, we show Theorem~\ref{L2regu-thm}.
Though we follow the idea of the proof of Theorem~1 in Jeganathan~\cite{jeg82},
some results for expectation and probability are replaced with those for conditional expectation and conditional probability.

We first prepare several lemmas.

\begin{lemma}\label{cond-exp-conv-lemma}
Let $(\mathcal{F}_{n,j})_{j=0}^{N_n}$ be a filtration on some probability space $(\Omega_n,\mathcal{F}_n,P_n)$ for $n\in\mathbb{N}$.
Let $X_{n,j}$ be a nonnegative-valued, $\mathcal{F}_{n,j}$-measurable random variable for $1\leq j\leq N_n$ and $n\in\mathbb{N}$.
Then,
\begin{enumerate}
\item $\sum_{j=1}^{N_n}E_n[X_{n,j}|\mathcal{F}_{n,j-1}]\overset{P_n}\to 0$ as $n\to \infty$ implies that $\sum_{j=1}^{N_n}X_{n,j}\overset{P_n}\to 0$ as $n\to \infty$,
\item $P_n$-tightness of $\{\sum_{j=1}^{N_n}E_n[X_{n,j}|\mathcal{F}_{n,j-1}]\}_{n\in\mbbn}$ implies
$P_n$-tightness of $\{\sum_{j=1}^{N_n}X_{n,j}\}_{n\in\mbbn}$,
\end{enumerate}
where $E_n$ denotes the expectation with respect to $P_n$.
\end{lemma}

\begin{proof}
1. For any $\delta>0$, let $A_{n,j,\delta}=\{\sum_{j'=1}^jE_n[X_{n,j'}|\mathcal{F}_{n,j'-1}]\leq \delta\}$.
Since $A_{n,j,\delta}\in\mathcal{F}_{n,j-1}$
and $A_{n,j,\delta}$ is monotonically decreasing on $j$, we have
\begin{eqnarray}
E_n\bigg[\sum_{j=1}^{N_n}X_{n,j}1_{A_{n,N_n,\delta}}\bigg]
&\leq &E_n\bigg[\sum_{j=1}^{N_n}X_{n,j}1_{A_{n,j,\delta}}\bigg]
\leq E_n\bigg[\sum_{j=1}^{N_n}E_n[X_{n,j}|\mathcal{F}_{n,j-1}]1_{A_{n,j,\delta}}\bigg]. \nonumber
\end{eqnarray}
If $\omega \in A_{n,1,\delta}$, we have
\begin{equation}\label{cond-exp-conv-lemma-eq1}
\sum_{j=1}^{N_n}E_n[X_{n,j}|\mathcal{F}_{n,j-1}]1_{A_{n,j,\delta}}=\sum_{j=1}^JE_n[X_{n,j}|\mathcal{F}_{n,j-1}]\leq \delta,
\end{equation}
where $J=\max\{1\leq j\leq N_n;\omega \in A_{n,j,\delta}\}$.
Hence we obtain
\begin{equation}
E_n\bigg[\sum_{j=1}^{N_n}X_{n,j}1_{A_{n,N_n,\delta}}\bigg]\leq E_n\bigg[\bigg(\sum_{j=1}^{N_n}E_n[X_{n,j}|\mathcal{F}_{n,j-1}]\bigg)\wedge \delta\bigg]\to 0.
\end{equation}
Therefore, it follows from Markov's inequality that
\EQNN{P_n\bigg(\sum_{j=1}^{N_n}X_{n,j}\geq \delta\bigg)&\leq P_n(A_{n,N_n,\delta}^c)+P_n\bigg(\sum_{j=1}^{N_n}X_{n,j}1_{A_{n,N_n,\delta}}\geq \delta\bigg) \\
&\leq P_n(A_{n,N_n,\delta}^c)+\delta^{-1}E_n\bigg[\sum_{j=1}^{N_n}X_{n,j}1_{A_{n,N_n,\delta}}\bigg] \to 0.
}

\noindent
2. For any $\epsilon>0$, there exists $M>0$ such that $P_n(A_{n,N_n,M}^c)<\epsilon$ by the assumptions.
Then, (\ref{cond-exp-conv-lemma-eq1}) and the monotonicity of $A_{n,j,M}$ on $j$ yield
\EQNN{P_n\bigg(\sum_{j=1}^{N_n}X_{n,j}\geq M'\bigg)
&\leq P_n\bigg(\sum_{j=1}^{N_n}X_{n,j}1_{A_{n,N_n,M}}\geq M'\bigg) +P_n(A_{n,N_n,M}^c) \\
&\leq \frac{1}{M'}E_n\bigg[\sum_{j=1}^{N_n}X_{n,j}1_{A_{n,j,M}}\bigg]+\epsilon \\
&=\frac{1}{M'}E_n\bigg[\sum_{j=1}^{N_n}E_n[X_{n,j}|\mcf_{n,j-1}]1_{A_{n,j,M}}\bigg]+\epsilon
\leq \frac{M}{M'}+\epsilon\leq 2\epsilon
}
for sufficiently large $M'$.
\end{proof}

\begin{lemma}\label{jeg82-lemma1}
Assume (A3) and that $(\mct_n)_{n\in\mbbn}$ is $P_{\alpha_0,n}$-tight. Then
\begin{equation*}
{\colorp \sum_{j=1}^{m_n}h^{\top} \epsilon_n^\top\eta_j\eta_j^{\top}\epsilon_nh-h^{\top}\mct_nh}\to 0
\end{equation*}
as $n\to \infty$ in $P_{\alpha_0,n}$-probability.
\end{lemma}

\begin{proof}

For any $\epsilon>0$, (A3) and {\colorp Chebyshev's} inequality imply that $\sum_{j=1}^{m_n}P_{\alpha_0,n}(|h^\top \epsilon_n^\top\eta_{j}|>\epsilon|\mca_{j-1,n}) \to 0$ in $P_{\alpha_0,n}$-probability.
Then, together with Lemma~\ref{cond-exp-conv-lemma}, we have
\begin{equation}\label{max-eta-est}
1_{\{\max_{1\leq j\leq m_n}|h^\top \epsilon_n^\top\eta_{j}|>\epsilon\}}\leq \sum_{j=1}^{m_n}1_{\{|h^\top \epsilon_n^\top\eta_{j}|>\epsilon\}}\to 0
\end{equation}
in $P_{\alpha_0,n}$-probability. Therefore, we obtain $P_{\alpha_0,n}(\max_{1\leq j\leq m_n}|h^\top \epsilon_n^\top\eta_{j}|>\epsilon)\to 0$.

Then, we have the conclusion similarly to (A.1) in~\cite{fuk-ogi20} by using (A3), Lemma~\ref{cond-exp-conv-lemma}, and the $P_{\alpha_0,n}$-tightness of
$\{\sum_{j=1}^{m_n}E_{\alpha_0}[|h^\top \epsilon_n^\top\eta_{j}|^2|\mca_{j-1,n}]\}_n$ shown by the assumptions.

\end{proof}

Let $G_j=1_{\{p_j(\alpha_0)=0\}}$, and let $\dot{\eta}_{nj}(\alpha_0,h)= (1-G_j)(p_j(\alpha_h)^{1/2}p_j(\alpha_0)^{-1/2}-1)$.
\begin{lemma}\label{jeg82-lemma3}
Under (A1),
\begin{equation*}
\sum_{j=1}^{m_n}\int G_jp_j(\alpha_h)d\mu_{n,j}\to 0 \quad {\rm and}
\quad \sum_{j=1}^{m_n}\int G_j|h^\top \epsilon_n^\top\dot{\xi}_{nj}|^2d\mu_{n,j}\to 0
\end{equation*}
{\colorp as $n\to\infty$ in $P_{\alpha_0,n}$-probability} for any $h\in \mathbb{R}^d$.
\end{lemma}
\begin{proof}
We can show the results in the same manner as Lemma 3 in~\cite{jeg82}.
\end{proof}

\begin{lemma}\label{jeg82-lemma4}
Under (A1) and (A4),
\EQ{\label{xi-conv1} {\colorp \sum_{j=1}^{m_n}\int \bigg|\xi_{nj}^2(\alpha_0,h)-\frac{1}{4}|h^\top \epsilon_n^\top\dot{\xi}_{nj}|^2\bigg|d\mu_{n,j} \to  0}}
and
\EQ{\label{xi-conv2} {\colorp \sum_{j=1}^{m_n}E_{\alpha_0}\bigg[\bigg|\dot{\eta}_{nj}^2(\alpha_0,h)-\frac{1}{4}|h^\top \epsilon_n^\top{\colorp \eta_j}|^2\bigg|\bigg|\mca_{j-1,n}\bigg] \to  0}}
{\colorp as $n\to \infty$ in $P_{\alpha_0,n}$-probability for any $h\in\mbbr^d$.}
\end{lemma}
\begin{proof}
Similarly to (2.12) in~\cite{jeg82}, we have
\EQNN{&\sum_{j=1}^{m_n}\int \bigg|\xi_{nj}^2(\alpha_0,h)-\frac{1}{4}|h^\top \epsilon_n^\top\dot{\xi}_{nj}|^2\bigg|d\mu_{n,j} \\
&\quad \leq (1+\beta)\sum_{j=1}^{m_n}\int \bigg|\xi_{nj}(\alpha_0,h)-\frac{1}{2}h^\top \epsilon_n^\top\dot{\xi}_{nj}\bigg|^2d\mu_{n,j}
+\frac{1}{4\beta}\sum_{j=1}^{m_n}\int |h^\top \epsilon_n^\top\dot{\xi}_{nj}|^2d\mu_{n,j}
}
for any $\beta>0$.

Moreover, we have
\begin{equation}\label{xi-eta-eq}
\sum_{j=1}^{m_n}\int|h^\top \epsilon_n^\top{\colorp \dot{\xi}_{nj}}|^2d\mu_{n,j}
=\sum_{j=1}^{m_n}E_{\alpha_0}[|h^\top \epsilon_n^\top{\colorp \eta_j|^2|\mca_{j-1,n}]} +\sum_{j=1}^{m_n}\int G_j|h^\top \epsilon_n^\top\dot{\xi}_{nj}|^2d\mu_{n,j}.
\end{equation}
Together with (A1), the tightness of $\{\sum_{j=1}^{m_n}E_{\alpha_0}[|h^\top \epsilon_n^\top\eta_j|^2|{\colorp \mca_{j-1,n}}]\}_n$ by (A4), and Lemma~\ref{jeg82-lemma3},
we have (\ref{xi-conv1}).
Next, we show (\ref{xi-conv2}). For any $\beta>0$, (2.9) in~\cite{jeg82} yields
\begin{eqnarray}
&&\sum_{j=1}^{m_n}E_{\alpha_0}\bigg[\bigg|\dot{\eta}_{nj}^2(\alpha_0,h)-\frac{1}{4}|h^\top \epsilon_n^\top\eta_j|^2\bigg|\bigg|{\colorp \mca_{j-1,n}}\bigg] \nonumber \\
&&\quad \leq (1+\beta)\sum_{j=1}^{m_n}E_{\alpha_0}\bigg[\bigg|\dot{\eta}_{nj}(\alpha_0,h)-\frac{1}{2}h^\top \epsilon_n^\top\eta_j\bigg|^2\bigg|{\colorp \mca_{j-1,n}}\bigg]
+\frac{1}{4\beta}\sum_{j=1}^{m_n}E_{\alpha_0}[|h^\top \epsilon_n^\top\eta_j|^2|{\colorp \mca_{j-1,n}}]. \nonumber
\end{eqnarray}
We also have
\begin{equation}\label{eta-diff-conv}
\sum_{j=1}^{m_n}E_{\alpha_0}\bigg[\bigg|\dot{\eta}_{nj}(\alpha_0,h)-\frac{1}{2}h^\top \epsilon_n^\top\eta_j\bigg|^2\bigg|{\colorp \mca_{j-1,n}}\bigg]
\leq \sum_{j=1}^{m_n}\int [\xi_{nj}(\alpha_0,h)-\frac{1}{2}h^{\top}\epsilon_n^\top\dot{\xi}_{nj}(\alpha_0)]^2d\mu_{n,j}\to 0
\end{equation}
in $P_{\alpha_0,n}$-probability by (A1).

Then, by letting $\beta\to \infty$,
the tightness of $\{\sum_{j=1}^{m_n}E_{\alpha_0}[|h^\top \epsilon_n^\top\eta_j|^2|{\colorp \mca_{j-1,n}}]\}_n$ yields (\ref{xi-conv2}).

\end{proof}

\begin{lemma}\label{jeg82-lemma5}
Under (A1)--(A4),
\begin{equation*}
\bigg|\sum_{j=1}^{m_n}\dot{\eta}_{nj}^2(\alpha_0,h)-\frac{1}{4}h^\top \mct_nh\bigg|\to 0
\end{equation*}
as $n\to\infty$ in $P_{\alpha_0,n}$-probability.
\end{lemma}
\begin{proof}
Lemmas~\ref{jeg82-lemma1},~\ref{cond-exp-conv-lemma}, and~\ref{jeg82-lemma4} yield the desired result.
\end{proof}

\begin{lemma}\label{jeg82-lemma6}
Under (A1)--(A4),
\begin{equation*}
\max_{1\leq j\leq m_n}|\dot{\eta}_{nj}(\alpha_0,h)|\to 0 \quad {\rm and} \quad \sum_{j=1}^{m_n}|\dot{\eta}_{nj}(\alpha_0,h)|^3\to 0
\end{equation*}
in $P_{\alpha_0,n}$-probability.
\end{lemma}
\begin{proof}
We can show $\sum_{j=1}^{m_n}P_{\alpha_0,n}(|\dot{\eta}_{nj}(\alpha_0,h)|>\epsilon|{\colorp \mca_{j-1,n}})\to 0$ in $P_{\alpha_0,n}$-probability {\colorp for any $\epsilon>0$}
by (A3) and (\ref{eta-diff-conv}) similarly to Lemma 6 in~\cite{jeg82}.
Together with Lemma~\ref{cond-exp-conv-lemma}, we have
\begin{equation*}
1_{\{\max_{1\leq j\leq m_n}|\dot{\eta}_{nj}(\alpha_0,h)|>\epsilon\}}\leq \sum_{j=1}^{m_n}1_{\{|\dot{\eta}_{nj}(\alpha_0,h)|>\epsilon\}}\to 0
\end{equation*}
in $P_{\alpha_0,n}$-probability, which implies the first convergence.

The second convergence follows by
Lemma~\ref{jeg82-lemma5}, the first convergence, the {\colorp $P_{\alpha_0,n}$-tightness} of $\{\mct_n\}_{n=1}^{\infty}$, and the {\colorp inequality
$\sum_{j=1}^{m_n}|\dot{\eta}_{nj}(\alpha_0,h)|^3\leq\max_{1\leq j\leq m_n}|\dot{\eta}_{nj}(\alpha_0,h)|\times \sum_{j=1}^{m_n}|\dot{\eta}_{nj}(\alpha_0,h)|^2$}.

\end{proof}

\begin{lemma}\label{jeg82-lemma7}
Under (A1)--(A4),
\begin{equation*}
\bigg|2\sum_{j=1}^{m_n}\dot{\eta}_{nj}(\alpha_0,h)-h^\top \epsilon_n^\top\sum_{j=1}^{m_n}\eta_j+\frac{1}{4}h^\top \mct_nh\bigg|\to 0
\end{equation*}
in $P_{\alpha_0,n}$-probability.
\end{lemma}
\begin{proof}
(\ref{xi-eta-eq}) and Lemmas~\ref{jeg82-lemma4} and~\ref{jeg82-lemma3} yield
\begin{equation*}
\bigg|2\sum_{j=1}^{m_n}E_{\alpha_0}[\dot{\eta}_{nj}(\alpha_0,h)|{\colorp \mca_{j-1,n}}]+\frac{1}{4}h^\top \mct_nh\bigg| \to 0
\end{equation*}
in $P_{\alpha_0,n}$-probability similarly to Lemma 7 in~\cite{jeg82}.

Then by (A2), it is sufficient to show
$\sum_{j=1}^{m_n}[Y_j-E_{\alpha_0}[Y_j|{\colorp \mca_{j-1,n}}]]\to 0$ in $P_{\alpha_0,n}$-probability,
where $Y_j=2[\dot{\eta}_{nj}(\alpha_0,h)-h^\top \epsilon_n^\top\eta_j/2]$.

Since we have
\begin{equation*}
\sum_{j=1}^{m_n}E_{\alpha_0}[(Y_j-E_{\alpha_0}[Y_j|{\colorp \mca_{j-1,n}}])^2|{\colorp \mca_{j-1,n}}]\leq \sum_{j=1}^{m_n}E_{\alpha_0}[Y_j^2|{\colorp \mca_{j-1,n}}]\to 0
\end{equation*}
in $P_{\alpha_0,n}$-probability by (\ref{eta-diff-conv}),
Lemma 9 in Genon-Catalot and Jacod~\cite{gen-jac93} yields the conclusion.
\end{proof}

\noindent
{\bf Proof of Theorem~\ref{L2regu-thm}.}

We use a similar approach to Theorem~1 in~\cite{jeg82}.

For any $h\in\mathbb{R}^d$, Lemma~\ref{jeg82-lemma6} and Taylor's formula yield
\begin{eqnarray}
\log\frac{dP_{\alpha_h,n}}{dP_{\alpha_0,n}}
&=&2\sum_{j=1}^{m_n}\log(1+\dot{\eta}_{nj}(\alpha_0,h))
=2\sum_{j=1}^{m_n}\dot{\eta}_{nj}(\alpha_0,h)-\sum_{j=1}^{m_n}\dot{\eta}_{nj}^2(\alpha_0,h)
+\sum_{j=1}^{m_n}\beta_{nj}|\dot{\eta}_{nj}(\alpha_0,h)|^3 \nonumber
\end{eqnarray}
with probability tending to one, where $|\beta_{nj}|\leq 1$.

Together with Lemma~\ref{jeg82-lemma6}, we have
\begin{equation*}
\bigg|\log\frac{dP_{\alpha_h,n}}{dP_{\alpha_0,n}}
-2\sum_{j=1}^{m_n}\dot{\eta}_{nj}(\alpha_0,h)+\sum_{j=1}^{m_n}\dot{\eta}_{nj}^2(\alpha_0,h)\bigg|\to 0
\end{equation*}
in $P_{\alpha_0,n}$-probability. Therefore (A4) and Lemmas~\ref{jeg82-lemma5} and~\ref{jeg82-lemma7} yield the conclusion.

\qed

\section{The LAMN property via transition density approximation}\label{density-approx-section}

In this section, we give proofs of the results in Section~\ref{density-approx-subsection}.
{\colorp Let $K'_{n,j}=\bar{X}_j(K_{n,j})$.} \\

\noindent
{\bf Proof of Lemma \ref{B1-suff-lemma}.}

Let
\EQ{\label{delta-def} \delta_n=\max_{1\leq j\leq m_n}\sup_{\alpha\in \Theta,\bar{x}_{j-1}\in {\colorp K'_{n,j-1}}}|1-d_j (\bar{x}_{j-1},\alpha)|.}
Then, {\colorp by (\ref{p-diff-est}), we obtain  
\begin{equation}\label{dj-est}
\delta_n \leq \max_{1\leq j\leq m_n}\sup_{\bar{x}_{j-1}\in K'_{n,j-1}}\int |p_j(\alpha)-\tilde{p}_j(\alpha)|\mu_{n,j}(dx_j)=o(m_n^{-1}),
\end{equation}
and hence we have}
\EQ{\label{delta-est} (1-\delta_n)^{-m_n+1}=\exp(-(m_n-1)\log(1-\delta_n))=\exp(-(m_n-1)(-\delta_n+o(\delta_n)))\to 1}
{\colorp as $n\to\infty$. Then, for any $\epsilon>0$, the assumptions and the decomposition
\EQQ{\bigcup_{l=1}^{m_n-1}K_{n,l}^c=\bigsqcup_{l=1}^{m_n-1}\left[\left(\bigcap_{k=1}^{l-1} K_{n,k}\right)\bigcap K_{n,l}^c\right]}
yield}
\EQNN{\tilde{P}_{\alpha,n}(\cup_{l=1}^{m_n-1}K_{n,l}^c)
&=\sum_{l=1}^{m_n-1}\int \bigg(\prod_{j=1}^{m_n}\frac{\tilde{p}_j}{d_j}\bigg)1_{\cap_{k=1}^{l-1}K_{n,k}\cap K_{n,l}^c}\bigotimes_{j=1}^{m_n}\mu_{n,j}(dx_j) \\
&{\colorp = \sum_{l=1}^{m_n-1}\int \bigg(\prod_{j=1}^l\frac{\tilde{p}_j}{d_j}\bigg)1_{\bar{X}_l(\cap_{k=1}^{l-1}K_{n,k}\cap K_{n,l}^c)}\bigotimes_{j=1}^l\mu_{n,j}(dx_j)} \\
&\leq {\colorp \sum_{l=1}^{m_n-1}\frac{1}{(1-\delta_n)^l}\int \bigg(\prod_{j=1}^lp_j\bigg)1_{\bar{X}_l(\cap_{k=1}^{l-1}K_{n,k}\cap K_{n,l}^c)}\bigotimes_{j=1}^l\mu_{n,j}(dx_j)} \\
&\leq \frac{P_{\alpha,n}(\cup_{k=1}^{m_n-1}K_{n,k}^c)}{(1-\delta_n)^{m_n-1}}<\epsilon}
for any $\alpha$ and sufficiently large $n$.
\qed\\

\noindent
{\bf Proof of Theorem~\ref{log-likelihood-approx-thm}.}

First, (\ref{p-diff-est}) {\colorp and (\ref{dj-est}) imply that}
\begin{eqnarray}\label{dj-est2}
\sup_{\bar{x}_{j-1}\in K'_{n,j-1}}\int \bigg|p_j-\frac{\tilde{p}_j}{d_j}\bigg|\mu_{n,j}(dx_j)
\leq \sup_{\bar{x}_{j-1}\in K'_{n,j-1}}\bigg(|1-d_j^{-1}|+d_j^{-1}\int|p_j-\tilde{p}_j|\mu_{n,j}(dx_j)\bigg)=o(m_n^{-1}).
\end{eqnarray}
For any {\colorp $m\in\mbbn$} and nonnegative sequences {\colorp $\{a_j\}$ and $\{b_j\}$}, we can show that
\begin{equation}\label{pro-esti}
{\colorp \left|\prod_{j=1}^m a_j-\prod_{j=1}^m b_j\right|\leq \sum_{l=1}^m\left(\prod_{j=1}^{l-1} a_j\right)|a_l-b_l|\left(\prod_{j=l+1}^mb_j\right),}
\end{equation}
by induction.
Let $\mathcal{K}_n=\cap_{j=1}^{m_n-1}K_{n,j}\subset \mathcal{X}_n$. For any $\epsilon>0$, it follows from (B1), \eqref{dj-est2}, and \eqref{pro-esti} that
\EQNN{\lVert P_{\alpha,n}-\tilde{P}_{\alpha,n}\rVert
&\leq \lVert P_{\alpha,n}|_{\mathcal{K}_n}-\tilde{P}_{\alpha,n}|_{\mathcal{K}_n}\rVert+P_{\alpha,n}(\mathcal{K}_n^c)+\tilde{P}_{\alpha,n}(\mathcal{K}_n^c) \\
&\leq \int \bigg|\prod_{j=1}^{m_n}p_j-\prod_{j=1}^{m_n}\frac{\tilde{p}_j}{d_j}\bigg|1_{\mathcal{K}_n}{\colorp \bigotimes_{j=1}^{m_n}\mu_{n,j}(dx_j)}+2\epsilon \\
&\leq \sum_{l=1}^{m_n}\int \bigg(\prod_{j=1}^{l-1}p_j\bigg)\bigg|p_l-\frac{\tilde{p}_l}{d_l}\bigg|\bigg(\prod_{j=l+1}^{m_n}\frac{\tilde{p}_j}{d_j}\bigg)
{\colorp 1_{K_{n,l-1}}\bigotimes_{j=1}^{m_n}\mu_{n,j}(dx_j)}+2\epsilon \\
&\leq o(m_n^{-1})\times \sum_{l=1}^{m_n}\int \bigg(\prod_{j=1}^{l-1}p_j\bigg){\colorp \bigotimes_{j=1}^{l-1}\mu_{n,j}(dx_j)}+2\epsilon
=o(1)+2\epsilon,}
for any $\alpha$ and sufficiently large $n$, which implies $\sup_\alpha\lVert P_{\alpha,n}-\tilde{P}_{\alpha,n}\rVert\to 0$ as $n\to\infty$.
{\colorp If futher, (\ref{tildeP-cond}) is satisfied, Proposition 4.3.2 in Le Cam~\cite{lec86} and the discussion after that lead} to the conclusion.

\qed\\

\noindent
{\bf Proof of Theorem~\ref{density-approx-thm}.} Let
\EQ{\label{eta-def} \eta_j=\frac{\PA(\tilde{p}_j/d_j)|_{\alpha=\alpha_0}}{\tilde{p}_j(\alpha_0)/d_j(\bar{x}_{j-1},\alpha_0)}1_{\{\tilde{p}_j(\alpha_0)\neq 0\}}1_{K'_{n,j-1}},}
then we have
\EQ{\label{eta-eq} \epsilon_n^\top \eta_j={\colorp \tilde{\eta}_j}-\frac{(\PTT D_{j,e_1},\cdots, \PTT D_{j,e_d})}{d_j(\bar{x}_{j-1},\alpha_0)}1_{\{\tilde{p}_j(\alpha_0)\neq 0\}}1_{K'_{n,j-1}}.
}
We verify (A1)--(A4) in Theorem~\ref{L2regu-thm} for $\tilde{P}_{\alpha,n}$.
Fix $\epsilon>0$ such that {\colorp $\{\alpha_{t\epsilon e_i}\}_{t\in [0,1]}\subset \Theta$ and (\ref{zeta-integrable}) is satisfied with $h=e_i$ and $\delta=\epsilon$} for any $1\leq i\leq d$.
Then (\ref{zeta-integrable}) and Fubini's theorem yield
\EQN{\label{tildePderivaEst} &\int \sup_{t\in [0,1]}\bigg|\frac{d}{dt}\tilde{p}_j(\alpha_{t\epsilon e_i})\bigg|\mu_{n,j}(dx_j) \\
&\quad \leq \int \bigg|\frac{d}{dt}\tilde{p}_j(\alpha_{t\epsilon e_i})\bigg|_{t=0}\bigg|\mu_{n,j}(dx_j)
+\int \int^1_0\bigg|\bigg(\frac{d}{dt}\bigg)^2\tilde{p}_j(\alpha_{t\epsilon e_i})\bigg|dt\mu_{n,j}(dx_j) \\
&\quad \leq \int \bigg|\frac{d}{dt}\tilde{p}_j(\alpha_{t\epsilon e_i})\bigg|_{t=0}\bigg|\mu_{n,j}(dx_j)
+\sup_{t\in [0,1]}\int\bigg|\bigg(\frac{d}{dt}\bigg)^2\tilde{p}_j(\alpha_{t\epsilon e_i})\bigg|\mu_{n,j}(dx_j)<\infty
}
for $\bar{x}_{j-1}\in K'_{n,j-1}$ and $1\leq i\leq d$ under (B2).
Therefore, we obtain
\EQN{\label{eta-j-zero} \tilde{E}_{\alpha_0}[\eta_j|\mca_{j-1,n}]
&=(\epsilon_n^\top)^{-1}\int \frac{d}{dt}\bigg(\frac{\tilde{p}_j(\alpha_{te_i})}{D_{j,e_i}(\bar{x}_{j-1},t)}\bigg)\bigg|_{t=0}{\colorp 1_{\{\tilde{p}_j(\alpha_0)\neq 0\}}}\mu_{n,j}(dx_j)1_{K'_{n,j-1}} \\
&= (\epsilon_n^\top)^{-1}\frac{d}{dt}\int\frac{\tilde{p}_j(\alpha_{te_i})}{D_{j,e_i}(\bar{x}_{j-1},t)}\mu_{n,j}(dx_j)\bigg|_{t=0}1_{K'_{n,j-1}}=0}
{\colorp because zero points of $\tilde{p}_j$ do not depend on $\alpha$.}
We also have $E_{\alpha_0}[|\eta_j|^2|\mca_{j-1,n}]<\infty$ by (\ref{D-del-est}), (\ref{delta-est}), {\colorp (\ref{eta-eq}), and (B3) for sufficiently large $n$}, which implies (A2).

Next, we verify (A1).
Fix arbitrary $h\in\mbbr^d$. We have $\{\alpha_{th}\}_{t\in [0,1]}\subset \Theta$ and {\colorp $\delta_n$ defined in (\ref{delta-def}) satisfies} $\delta_n\geq 1/2$ for sufficiently large $n$. Let
\EQQ{f_{j,h}(t)=\frac{\tilde{p}_j(\alpha_{th})}{D_{j,h}(\bar{x}_{j-1},t)},}
then similarly to (2.6) in~\cite{fuk-ogi20}, a simple calculation yields
\EQNN{&\int \bigg[\xi_{nj}(\alpha_0,h)-\frac{1}{2}h^\top \epsilon_n^\top \dot{\xi}_{nj}(\alpha_0)\bigg]^2d\mu_{n,j}1_{K'_{n,j-1}} \\
&\quad \leq \frac{1}{4}\sup_{0\leq t\leq 1}\tilde{E}_{\alpha_{th}}\bigg[\bigg|\frac{\partial_t^2f_{j,h}}{f_{j,h}}(t)-\frac{(\partial_t f_{j,h})^2}{2f_{j,h}^2}(t)\bigg|^2{\colorp 1_{\{f_{j,h}(t)\neq 0\}}}\bigg|\mca_{j-1,n}\bigg]1_{K'_{n,j-1}},}
{\colorp where $\dot{\xi}_{nj}(\alpha_0)=\eta_j\sqrt{f_{j,h}(0)}$.}
Moreover, since
\EQNN{\frac{\partial_t f_{j,h}}{f_{j,h}}{\colorp 1_{\{f_{j,h}\neq 0\}}}&=\zeta_{j,t}^{1,h}-\frac{\partial_t D_{j,h}}{D_{j,h}}(\bar{x}_{j-1},t){\colorp 1_{\{f_{j,h}\neq 0\}}}, \\
\frac{\partial_t^2 f_{j,h}}{f_{j,h}}{\colorp 1_{\{f_{j,h}\neq 0\}}}&=\zeta_{j,t}^{2,h}-2\zeta_{j,t}^{1,h}\frac{\partial_t D_{j,h}}{D_{j,h}}
+\bigg(2\frac{(\partial_t D_{j,h})^2}{D_{j,h}^2}-\frac{\partial_t^2D_{j,h}}{D_{j,h}}\bigg){\colorp 1_{\{f_{j,h}\neq 0\}}},
}
(B3) and (B2) imply
\EQN{\label{L2-regu-ineq} &\sum_{j=1}^{m_n}\sup_{t\in[0,1]}\tilde{E}_{\alpha_{th}}\bigg[\bigg\{\frac{\partial_t^2 f_{j,h}}{f_{j,h}}1_{\{f_{j,h}\neq 0\}}\bigg\}^2\bigg|\mca_{j-1,n}\bigg]1_{\cap_{j'}K_{n,j'}} \\
&\quad \leq C\sum_{j=1}^{m_n}\sup_{t\in[0,1]}\bigg\{\tilde{E}_{\alpha_{th}}[|\zeta_{j,t}^{2,h}|^2|\mca_{j-1,n}]
+ \tilde{E}_{\alpha_{th}}[|\zeta_{j,t}^{1,h}|^2|\mca_{j-1,n}]\cdot \bigg|\frac{\partial_t D_{j,h}}{D_{j,h}}\bigg|^2
+\bigg|\frac{\partial_t D_{j,h}}{D_{j,h}}\bigg|^4
+\bigg|\frac{\partial_t^2 D_{j,h}}{D_{j,h}}\bigg|^2\bigg\}1_{\cap_{j'}K_{n,j'}} \\
&\quad \leq o_p(1)+C\bigg\{\sum_{j=1}^{m_n}\sup_{t\in[0,1]}\tilde{E}_{\alpha_{th}}[|\zeta_{j,t}^{1,h}|^4|\mca_{j-1,n}]\bigg\}^{\frac{1}{2}}
\bigg\{\sum_{j=1}^{m_n}\sup_{t\in[0,1]}\bigg|\frac{\partial_t D_{j,h}}{D_{j,h}}\bigg|^4\bigg\}^{\frac{1}{2}}1_{\cap_{j'}K_{n,j'}} + o_p(1) \TOP 0.
}
Similarly, {\colorp we have}
\EQ{\label{f4-est} \sum_{j=1}^{m_n}\sup_{t\in[0,1]}\tilde{E}_{\alpha_{th}}{\colorp \bigg[\bigg|\frac{\PTT f_{j,h}}{f_{j,h}}\bigg|^41_{\{f_{j,h}\neq 0\}}\bigg|\mca_{j-1,n}\bigg]1_{\cap_{j'}K_{n,j'}}}\TOP 0.}
Together with (B1), (A1) holds. We can similarly show (A3).

Furthermore, because
\begin{equation*}
\sum_{j=1}^{m_n}\bigg(\frac{{\colorp \PTT} D_{j,h}}{D_{j,h}}\bigg)^k1_{\cap_{j'}K_{n,j'}}\to^{P_{\alpha_0,n}} 0, \quad \sum_{j=1}^{m_n}\frac{\partial_t D_{j,h}}{D_{j,h}}\tilde{E}_{\alpha_0}[\zeta_{j,0}^{1,h}|\mca_{j-1,n}]1_{\cap_{j'}K_{n,j'}}\to^{P_{\alpha_0,n}} 0
\end{equation*}
for {\colorp $k\in \{1,2\}$, together with (\ref{eta-eq}) and (B1),} we have
\EQ{\label{B4-conv} \bigg(\epsilon_n^\top\sum_{j=1}^{m_n}\eta_j,\sum_{j=1}^{m_n}\epsilon_n^\top\tilde{E}_{\alpha_0}[\eta_j\eta_j^\top|\mca_{j-1,n}]\epsilon_n\bigg)
=(\tilde{V}_n,\tilde{\mct}_n)+o_p(1),
}
and hence (A4) holds by {\colorp (B4)}.
Therefore, Condition (L) for $\{\tilde{P}_{\alpha,n}\}_{\alpha,n}$ holds by Theorem~\ref{L2regu-thm}.
Together with Theorem~\ref{log-likelihood-approx-thm}, we have Condition (L)
for $\{P_{\alpha,n}\}_{\alpha,n}$.
Then Remark~\ref{LAMN-remark} yields the LAMN property when $\mct$ is positive definite almost surely and $\epsilon_n$ is {\colorp symmetric and} positive definite for any $n$.

\qed

\noindent
{\bf Proof of Corollary~\ref{B4'-cor}.}\\

It is sufficient to verify {\colorp (A4) for
$\{\tilde{P}_{\alpha,n}\}_{\alpha,n}$ and $\eta_j$ defined by (\ref{eta-def})}.
(\ref{B4-conv}) and (B4$'$) imply
\EQ{\label{density-LAMN-cor-eq1} \epsilon_n^\top\sum_{j=1}^{m_n}\tilde{E}_{\alpha_0}[\eta_j\eta_j^\top|{\colorp \mca_{j-1,n}}]\epsilon_n \TOP \mct.}
Moreover, by setting $t=0$, {\colorp (\ref{f4-est}) and (B1) yield}
\EQ{\label{density-LAMN-cor-eq2} \sum_{j=1}^{m_n}\tilde{E}_{\alpha_0}[|\epsilon_n^\top\eta_j|^4|{\colorp \mca_{j-1,n}}]\TOP 0.}

{\colorp Therefore,} (\ref{eta-j-zero}), (\ref{density-LAMN-cor-eq2}), {\colorp (\ref{density-LAMN-cor-eq1}),} and a martingale central limit theorem (see Cor~3.1 and the following remark in Hall and Hyde~\cite{hal-hey80} for example) yield
\EQQ{\mcl\bigg(\epsilon_n^\top\sum_{j=1}^{m_n}\eta_j\bigg|P_{\alpha_0,n}\bigg)\to N(0,\mct),}
which implies (A4).

\qed

\section{Proof of the LAN property for jump-diffusion processes}\label{jump-LAN-section}

In this section, we show the LAN property of jump-diffusion processes based on the scheme proposed in Section~\ref{density-approx-subsection}.
We approximate the genuine likelihood by a thresholding likelihood that can roughly distinguish whether the increments contain at least one jump or not.
We introduce some conventions used in the rest of this paper.
\begin{itemize}
\item For a matrix $A$ and a vector $v$, we denote element $(i,j)$ of $A$ by $[A]_{ij}$ and element $i$ of $v$ by $[v]_i$. We often regard an $r$-dimensional vector $v$ as an $r\times 1$ matrix.
\item $C$ and $C_p$ denote generic positive constants whose values may vary depending on context.
\end{itemize}

\medskip

{\colorp To} deal with the continuous part of $X$, we briefly review the results of {\colorp Nualart~\cite{nua06} and} Gobet~\cite{gob01,gob02} in the following section.

\subsection{Results for continuous part}\label{conti-part-subsection}

We define the stochastic process $(X^{\alpha,c}_t)_{t\geq 0}=(X^{\alpha,c}_{t,x})_{t\geq 0}$ by the solution of
the stochastic differential equation: $X_0^{\alpha,c}=x$ and
\EQQ{dX_t^{\alpha,c}=a(X_t^{\alpha,c},\theta)dt+b(X_t^{\alpha,c},\sigma)dW_t.}
Then, {\colorp under (C1),} Theorem~39 in Chapter V of Protter~\cite{pro90} ensures the existence of $\partial_\alpha^lX_t^{\alpha,c}$ for $l\in \{1,2\}$, and we have
\EQQ{\partial_\alpha X_t^{\alpha,c}=\int^t_0(\partial_\alpha a(X_s^{\alpha,c},\theta)+\partial_xa(X_s^{\alpha,c},\theta)\partial_\alpha X_s^{\alpha,c})ds
+\int^t_0(\partial_\alpha b(X_s^{\alpha,c},\sigma)+\partial_xb(X_s^{\alpha,c},\sigma)\partial_\alpha X_s^{\alpha,c})dW_s,
}
{\colorp \EQNN{
&\partial_{\alpha_i}\partial_{\alpha_j} X_t^{\alpha,c} \\
&\quad =\int^t_0\bigg\{\partial_{\alpha_i}\partial_{\alpha_j}a+\sum_k(\partial_{x_k}\partial_{\alpha_i} a\mcz^{j,k}_s+\partial_{x_k}\partial_{\alpha_j} a\mcz^{i,k}_s+\partial_{x_k} a\mcz_s^{i,j,k})+\sum_{k,l}\partial_{x_k}\partial_{x_l}a\mcz^{i,k}_s\mcz^{j,l}_s\bigg\}(X_s^{\alpha,c},\theta)ds \\
&\quad \quad +\int^t_0\bigg\{\partial_{\alpha_i}\partial_{\alpha_j}b+\sum_k(\partial_{x_k}\partial_{\alpha_i} b\mcz_s^{j,k}+\partial_{x_k}\partial_{\alpha_j} b\mcz_s^{i,k}+\partial_{x_k}b \mcz_s^{i,j,k})+\sum_{k,l}\partial_{x_k}\partial_{x_l}b\mcz_s^{i,k}\mcz_s^{j,l}\bigg\}(X_s^{\alpha,c},\sigma)dW_s,
}
where $\mcz^{j,k}_s=\partial_{\alpha_j} [X_s^{\alpha,c}]_k$ and $\mcz^{i,j,k}_s=\partial_{\alpha_i}\partial_{\alpha_j} [X_s^{\alpha,c}]_k$. }
\begin{discuss}
{\colorr ガウスの時\EQQ{(nh_n)^{-2}\sum_j\partial_\theta a_jS_j^{-1}(\Delta_j X-h_na_j)=O(nh_n^2/(nh_n)^2).}}
\end{discuss}

Let {\colorp $T\in (0,1]$}. Together with Theorem~2.2.1 and Lemma 2.2.2 in Nualart~\cite{nua06}, (C1), and Gronwall's inequality, we have
\EQN{\label{delX-est} E[|\partial_\sigma^lX_{t,x}^{\alpha,c}|^p]^{1/p}\leq C_p\sqrt{T}(1+|x|)^{C_p},
\quad E[|\PT^{1+l_1}\PS^{l_2}X_{t,x}^{\alpha,c}|^p]^{1/p}\leq C_p T(1+|x|)^{C_p}, \\
E\Big[\sup_{r\in [0,t]}|D_r\PS^{l_3}X_{t,x}^{\alpha,c}|^p\Big]^{1/p}\leq C_p(1+|x|)^{C_p},
\quad E\Big[\sup_{r\in [0,t]}|D_r\PT X_{t,x}^{\alpha,c}|^p\Big]^{1/p}\leq C_pT(1+|x|)^{C_p}, \\
\lVert \PS^{l_3}X_{t,x}^{\alpha,c}\rVert_{3-{\colorp l_3},p}\leq C_p(1+|x|)^{C_p},
\quad \lVert \PT^lX_{t,x}^{\alpha,c}\rVert_{3-l,p}\leq C_pT(1+|x|)^{C_p},
\quad \lVert \PS\PT X_{t,x}^{\alpha,c}\rVert_{1,p}\leq C_pT(1+|x|)^{C_p}
}
for some positive constant $C_p$ that depends on only $p$ and for $0\leq t\leq T$, $l\in \{1,2\}$, $(l_1,l_2)\in \{(0,0),(1,0),(0,1)\}$, and $l_3\in \{0,1,2\}$,
where $D$ is the Malliavin--Shigekawa derivative related to the underlying Hilbert space $H=L^2([0,{\colorp T}];\mbbr^m)$, and $\lVert \cdot \rVert_{k,p}$ denotes the seminorm defined in (1.37) in~\cite{nua06}.

We define an $m\times m$ matrix-valued process $\mathcal{U}_t^\alpha$ by
\begin{equation}
[\mathcal{U}_t^\alpha]_{ij}=\delta_{ij}+\sum_{k=1}^m\int^t_0\partial_{x_k} [a]_i(X_s^{\alpha,c},\theta)[\mathcal{U}_s^\alpha]_{kj} ds
+\sum_{k,l=1}^m\int^t_0\partial_{x_k} [b]_{il}(X_s^{\alpha,c},\sigma)[\mathcal{U}_s^\alpha]_{kj} d[W_s]_l,
\end{equation}
where $\delta_{ij}$ is the Kronecker delta.
Then, by the argument in Section 2.3.1 of Nualart~\cite{nua06}, $\mathcal{U}_t^\alpha$ is invertible and we have
\begin{eqnarray}
[(\mathcal{U}_t^\alpha)^{-1}]_{ij}
&=&\delta_{ij}-\sum_{k=1}^m\int^t_0[(\mathcal{U}_s^\alpha)^{-1}]_{ik}\bigg(\partial_{x_j} [a]_k(X_s^{\alpha,c},\theta)-\sum_{l,p=1}^m\partial_{x_p} [b]_{kl}(X_s^{\alpha,c},\sigma)\partial_{x_j} [b]_{pl}(X_s^{\alpha,c},\sigma)\bigg)ds \nonumber \\
&&-\sum_{k,l=1}^m\int^t_0[(\mathcal{U}_s^\alpha)^{-1}]_{ik}\partial_{x_j} [b]_{kl}(X_s^{\alpha,c},\sigma) d[W_s]_l. \nonumber
\end{eqnarray}

Lemma 2.2.2 in Nualart~\cite{nua06} yields
\EQ{\label{U-est} \lVert \mcu^\alpha\rVert_{2,p}\vee \lVert (\mcu^\alpha)^{-1}\rVert_{2,p}\leq C_p(1+|x|)^{C_p}.}

Conditions (C1) and (C2) and Theorem~2.3.1 in~\cite{nua06} ensure existence of the density function of $X^{\alpha,c}_{t,x}$.
Let us denote this density by $p_{x,\alpha}^{c,t}(y)$.
{\colorp Let $\delta$ be the Hitsuda-Skorokhod integral (the divergence operator).}

The following proposition is Proposition 2.2 in Gobet~\cite{gob02} and Lemma 3.6 in Ogihara~\cite{ogi15}. The assumptions for the original ones are different from this proposition. However, we can verify from their proofs that the results hold under (C1) and (C2).
\begin{proposition}\label{gob-prop2}
Assume (C1) and (C2) and set $T\in (0,1]$. Let $U_{l,t}^{\alpha,T}=([b^{-1}(X_t^{\alpha,c},\sigma)\mathcal{U}_t^\alpha(\mathcal{U}_T^\alpha)^{-1}]_{lk})_{k=1}^m$, {\colorp and $U_t^{\alpha,T}=(U_{1,t}^{\alpha,T} ~\cdots~ U_{m,t}^{\alpha,T})^\top$} for $t\in [0,T]$. Then
\begin{eqnarray}
\frac{\partial_\alpha p_{x,\alpha}^{c,T}}{p_{x,\alpha}^{c,T}}(y)&=&E_x\bigg[T^{-1}\delta((U^{\alpha,T})^\top\partial_\alpha X_T^{\alpha,c})\bigg|X_T^{\alpha,c}=y\bigg], \nonumber \\
\frac{\partial_\alpha^2p_{x,\alpha}^{c,T}}{p_{x,\alpha}^{c,T}}(y)
&=&E_x\bigg[T^{-1}\delta((U^{\alpha,T})^\top\partial_\alpha^2X_T^{\alpha,c})
+T^{-2}\sum_{i=1}^m\delta(U_i^{\alpha,T}\delta((U^{\alpha,T})^\top\partial_\alpha X_T^{\alpha,c}\partial_\alpha [X_T^{\alpha,c}]_i))\bigg|X_T^{\alpha,c}=y\bigg]. \nonumber
\end{eqnarray}
\end{proposition}
\begin{discuss}
{\colorr $a,b$に$t$を入れるには$X^{\alpha,c}$, $p^{c,T}_x$のnotationを少し変える必要があるが, densityの存在定理なども引き続き使える．しかしエルゴード性の部分がだめで極限が書けないか}
\end{discuss}

\begin{proposition}\label{pc-est}
Assume (C1) and (C2). Then for any $p\geq 2$, there exists a positive constant $C_p$ such that
\begin{eqnarray}
\int \bigg|\frac{\partial_\theta^l p_{x,\alpha}^{c,T}}{p_{x,\alpha}^{c,T}}\bigg|^pp_{x,\alpha}^{c,T}(y)dy&\leq &C_pT^{p/2}(1+|x|)^{C_p}, \label{pc-est-lemma-eq1} \\
\int\bigg|\frac{\PS^l p_{x,\alpha}^{c,T}}{p_{x,\alpha}^{c,T}}\bigg|^pp_{x,\alpha}^{c,T}(y)dy&\leq &C_p(1+|x|)^{C_p}, \label{pc-est-lemma-eq2} \\
\int\bigg|\frac{\PS\PT p_{x,\alpha}^{c,T}}{p_{x,\alpha}^{c,T}}\bigg|^pp_{x,\alpha}^{c,T}(y)dy&\leq &C_pT^{p/2}(1+|x|)^{C_p} \label{pc-est-lemma-eq3}
\end{eqnarray}
for any $\alpha\in\Theta$, $x\in \mbbr^m$, $T \in (0,1]$, and $l\in \{1,2\}$.
\end{proposition}

\begin{proof}
Point (iii) of Proposition 3.2 in Gobet~\cite{gob01} yields
\EQ{\label{pc-est-lemma-eq4} \delta(\PA [X_T^{\alpha,c}]_iU_i^{\alpha,T})=\PA [X_T^{\alpha,c}]_i\delta(U_i^{\alpha,T})
-\int_0^T D_u\PA [X_T^{\alpha,c}]_i\cdot U_{i,u}^{\alpha,T}du.}

Moreover, the Clark--Ocone formula (Proposition 3.3 in~\cite{gob01}) yields
\EQN{\label{pc-est-lemma-eq5} E[|\delta(U_i^{\alpha,T})|^p]
&=E\bigg[\bigg|\int^T_0E[D_u\delta(U_i^{\alpha,T})|\mcf_u]\cdot dW_u\bigg|^p\bigg] \\
&\leq C_pT^{p/2-1}\int^T_0E[|D_u\delta(U_i^{\alpha,T})|^p]du.}

Points (iv) and (v) of Proposition 3.2 in~\cite{gob01} yield
\EQ{\label{pc-est-lemma-eq6} D_u\delta(U_i^{\alpha,T})=U_i^{\alpha,T}+\delta(D_uU_i^{\alpha,T}),}
and
\EQN{\label{pc-est-lemma-eq7} E[|\delta (D_uU_i^{\alpha,T})|^p]
\leq {\colorp C_p}E\bigg[\int^T_0 |D_uU_{i,v}^{\alpha,T}|^pdv\bigg]
+{\colorp C_p}E\bigg[\int_0^T\int_0^T|D_{v_1}D_uU_{i,v_2}^{\alpha,T}|^pdv_1dv_2\bigg]
\leq C_p(1+|x|)^{C_p}.}

(\ref{delX-est}), (\ref{U-est}), and (\ref{pc-est-lemma-eq5})--(\ref{pc-est-lemma-eq7}) yield
\EQ{\label{pc-est-lemma-eq8} E[|\delta(U_i^{\alpha,T})|^p]\leq C_pT^{p/2}(1+|x|)^{C_p}.}
Then (\ref{delX-est}), (\ref{U-est}), (\ref{pc-est-lemma-eq4}) and (\ref{pc-est-lemma-eq8}) yield
\EQQ{E[|\delta(\PT [\XC_T]_iU_i^{\alpha,T})|^p]\leq C_pT^{3p/2}(1+|x|)^{C_p},
\quad E[|\delta(\PS [\XC_T]_iU_i^{\alpha,T})|^p]\leq C_pT^p(1+|x|)^{C_p}.}
Together with Proposition~\ref{gob-prop2}, we have (\ref{pc-est-lemma-eq1}) and (\ref{pc-est-lemma-eq2}) with $l=1$.

For the estimates for $\PA^2 {\colorp p_{x,\alpha}^{c,T}}$, we first obtain
\EQN{\label{pc-est-prop-eq1} &\delta(U_i^{\alpha,T}\delta((U^{\alpha,T})^\top\partial_\alpha X_T^{\alpha,c}\partial_\alpha [X_T^{\alpha,c}]_i)) \\
&\quad =\delta(U^{\alpha,T})\delta((U^{\alpha,T})^\top\partial_\alpha X_T^{\alpha,c}\partial_\alpha [X_T^{\alpha,c}]_i)
-\int^T_0D_t\delta((U^{\alpha,T})^\top\partial_\alpha X_T^{\alpha,c}\partial_\alpha [X_T^{\alpha,c}]_i))\cdot U_{i,t}^{\alpha,T}dt \\
&\quad =\delta(U^{\alpha,T})\left\{\sum_k\delta( U^{\alpha,T}_k)\partial_\alpha [X_T^{\alpha,c}]_k\partial_\alpha [X_T^{\alpha,c}]_i-\sum_k\int_0^T U^{\alpha,T}_{k,t}\cdot D_t(\partial_\alpha [X_T^{\alpha,c}]_k\partial_\alpha [X_T^{\alpha,c}]_i)dt\right\} \\
&\quad \quad -\int^T_0D_t\delta((U^{\alpha,T})^\top\partial_\alpha X_T^{\alpha,c}\partial_\alpha [X_T^{\alpha,c}]_i))\cdot U_{i,t}^{\alpha,T}dt.
}

The $L^p$ norm of the first term in the right-hand side is bounded by $C_pT^{5/2}(1+|x|)^{C_p}$ for $\PT^2$ and $\PT\PS$, and by $C_pT^2(1+|x|)^{C_p}$ for $\PS^2$ because of (\ref{delX-est}), (\ref{U-est}), and (\ref{pc-est-lemma-eq8}).
For the second term in the right-hand side of (\ref{pc-est-prop-eq1}), we have
\EQNN{
&-\int^T_0D_t\delta((U^{\alpha,T})^\top\partial_\alpha X_T^{\alpha,c}\partial_\alpha [X_T^{\alpha,c}]_i))\cdot U_{i,t}^{\alpha,T}dt \\
&\quad =-\int^T_0D_t\bigg\{\sum_k\bigg(\delta(U^{\alpha,T}_k)\PA [X_T^{\alpha,c}]_k\PA [X_T^{\alpha,c}]_i-\int^T_0D_s(\PA [X_T^{\alpha,c}]_k\PA [X_T^{\alpha,c}]_i)\cdot U_{k,s}^{\alpha,T}ds\bigg)\bigg\}\cdot U_{i,t}^{\alpha,T}dt \\
&\quad =-\sum_k\int^T_0\bigg\{(U^{\alpha,T}_{k,t}+\delta(D_tU^{\alpha,T}_k))\PA [X_T^{\alpha,c}]_k\PA [X_T^{\alpha,c}]_i
+\delta(U_k^{\alpha,T})D_t(\PA [X_T^{\alpha,c}]_k\PA [X_T^{\alpha,c}]_i) \\
&\quad \quad\quad \quad\quad \quad -\int^T_0D_t\left(D_s(\PA [X_T^{\alpha,c}]_k\PA [X_T^{\alpha,c}]_i)\cdot U_{k,s}^{\alpha,T}\right)ds\bigg\}\cdot U_{i,t}^{\alpha,T}dt.
}
Together with Proposition~\ref{gob-prop2}, (\ref{delX-est}), (\ref{U-est}), and (\ref{pc-est-lemma-eq8}), we have (\ref{pc-est-lemma-eq1})--(\ref{pc-est-lemma-eq3}) with $l=2$.

\end{proof}

\begin{remark}\label{third-derivative-rem}
Similarly to the proofs of Propositions~\ref{gob-prop2} and~\ref{pc-est}, we can show
\EQQ{\sup_{\alpha\in \Theta}\int \bigg|\frac{\PT^l\PS^kp_{x,\alpha}^{c,T}}{p_{x,\alpha}^{c,T}}\bigg|^pp_{x,\alpha}^{c,T}(y)dy\leq C_{p,n}(1+|x|)^{C_{p,n}}}
for $n\in\mbbn$, $p\geq 2$, $T\in (0,1]$, $x\in \mbbr^m$, and $l+k\leq 3$,
where $C_{p,n}$ is a positive constant depending on $p$ and $n$.
This result will be used in Sections~\ref{B1B2-subsection} and~\ref{B3B4-subsection}.
\end{remark}

The following estimates for transition density functions of diffusion processes are Proposition 1.2 in Gobet~\cite{gob02}.
\begin{proposition}\label{gob-prop}
Assume (C1) and (C2). Then there exist constants $c>1$ and $K>1$ such that
\begin{eqnarray}
p_{x,\alpha}^{c,t}(y)\leq Kt^{-m/2}\exp(-c^{-1}t^{-1}|x-y|^2+ct|x|^2), \\
p_{x,\alpha}^{c,t}(y)\geq K^{-1}t^{-m/2}\exp(-ct^{-1}|x-y|^2-ct|x|^2)
\end{eqnarray}
for $0<t\leq 1$, $x,y\in \mbbr^m$, and $\alpha\in\Theta$.
\end{proposition}

\subsection{Verifying Conditions (B1) and (B2)}\label{B1B2-subsection}

From now on, we show the LAN property of jump-diffusion processes by applying Corollary~\ref{B4'-cor}.
In this section, we first introduce our approximating likelihood function, and we check the conditions (B1) and (B2) for the function.

By dividing events, $p_j(x_{j-1},x_j,\al)$, the density function of $P(X_{t_j}^\alpha\in\cdot|\XJ^\alpha=x_{j-1})$ can be written as
\begin{eqnarray}
p_j(x_{j-1},x_j,\al)&=&p^0_j(x_{j-1},x_j,\alpha)+p^1_j(x_{j-1},x_j,\alpha)+\sum_{l=2}^\infty  p^2_{l,j}(x_{j-1},x_j,\alpha),
\end{eqnarray}
where $t_j=jh_n$,
\begin{eqnarray}
p^0_j(x_{j-1},x_j,\alpha)&=&e^{-\lambda h_n}p^{c,t_j-t_{j-1}}_{x_{j-1},\alpha}{\colorp (x_j)},\nonumber\\
p^1_j(x_{j-1},x_j,\alpha)&=&\lambda e^{-\lambda h_n}\int^{t_j}_{t_{j-1}}\int\int p^{c,\tau-t_{j-1}}_{x_{j-1},\alpha}(x)F_\theta(y)p_{x+y,\alpha}^{c,t_j-\tau}(x_j)dxdy d\tau, \nonumber \\
p^2_{l,j}(x_{j-1},x_j,\alpha)&=&\frac{\lambda^le^{-\lambda h_n}}{l!}\int^{t_j}_{t_{j-1}}\cdots \int^{t_j}_{t_{j-1}}\int\cdots \int p_{x_{j-1},\alpha}^{c,\tilde{\tau}_1-t_{j-1}}(z_1)F_\theta(z_2)\cdots p_{z_{2l-1}+z_{2l},\alpha}^{c,t_j-\tilde{\tau}_l}(x_j)\bigg(\prod_{j=1}^{2l}dz_j\bigg)\bigg(\prod_{k=1}^ld\tau_k\bigg). \nonumber
\end{eqnarray}
Here $\tilde{\tau}_1,\cdots,\tilde{\tau}_l$ are the sort of $\tau_1,\cdots, \tau_l$ in ascending order.
We can ignore the density function $p_0(x_0,\alpha)$ of $P(X_0^\alpha\in \cdot )$ {\colorp when we apply Corollary~\ref{B4'-cor}} because the distribution of $X_0^\alpha$ does not depend on $\alpha$ by the assumption.
Let $\epsilon_n={\rm diag}(n^{-1/2}I_{d_1},(nh_n)^{-1/2}I_{d_2}))$ and $T_n=nh_n$.

From (C4), there exist $\rho\in (1/4,1/2)$ and $\eta'>0$ such that $n^{1+\eta'}h_n^{1+(m+\gamma)\rho}\to 0$.
We write $L_n=\{{\colorp x\in \mbbr^m}||x|\leq h_n^\rho\}$.
For $\rho\in (0,1/2)$, Shimizu and Yoshida~\cite{shi-yos06} constructed a thresholding quasi-likelihood function based on the jump detection rule: $|X_{t_j}-\XJ|> h_n^\rho$ or not.
Here we follow their way.
More specifically, we approximate the genuine density $p_j$ by the thresholding quasi-likelihood function:
\begin{equation}
{\colorp \tilde{p}_j(\alpha)=\tilde{p}_j(\alpha,x_j,x_{j-1})=p^0_j(x_{j-1},x_j,\alpha)1_{L_n}(\Delta x_j)+p^1_j(x_{j-1},x_j,\alpha)1_{L_n^c}(\Delta x_j),}
\end{equation}
{\colorp where $\Delta x_j=x_j-x_{j-1}$,} and we apply Corollary~\ref{B4'-cor} to this function.
In this setting, $d_j=\int \tilde{p}_jdy\leq 1$ and Proposition~\ref{gob-prop} ensures $d_j>0$
{\colorp under (C1) and (C2).}

First we observe (B1).
For a constant $\delta\in(0,1/4)$,
we define
\EQQ{{\colorp K_{n,j}=\left\{(x_l)_{l=0}^n\subset \mbbr^{m(n+1)}\middle|\max_{0\leq l\leq j}|x_l|\leq n^\delta\right\} 
\quad {\rm and} \quad K''_{n,j}=\{x_j\in\mbbr^m||x_j|\leq n^\delta\}.
}}
For this set, \eqref{tail-prob-est1} follows from the following lemma.

\begin{lemma}\label{yu:maxesti}
Assume (C3). Then for any $\epsilon,\delta>0$, there exists a positive integer $N$ such that
\begin{equation*}
P\bigg[{\colorp \max_{0\leq k\leq n}}|{\colorp X_{t_k}^\alpha}|>n^\delta\bigg]<\epsilon
\end{equation*}
for all $n\geq N$ {\colorp and $\alpha\in\Theta$}.
\end{lemma}

\begin{proof}
Pick a positive constant $q$ fulfilling $q\delta >1$.
Then Chebyshev's inequality gives
\begin{equation*}
P\bigg[\max_k|X_{t_k}^\alpha|>n^\delta\bigg]\leq \frac{1}{n^{q\delta}}E\left[\max_k|X_{t_k}^\alpha|^q\right]\leq n^{1-q\delta}\sup_{t,\alpha}E[|X_t^\alpha|^q]\to 0
\end{equation*}
as $n\to \infty$.
\end{proof}

{\colorp Under (C1)--(C4),} Proposition~\ref{gob-prop} implies that
\EQNN{\int |p_j-\tilde{p}_j|dx_j
&=\int p_j^01_{L_n^c}({\colorp \Delta x_j})dx_j+P({\colorp N_\theta}((t_{j-1},t_j]\times E)\geq 2)+\int p^1_j1_{L_n}(\Delta x_j)dx_j \\
&\leq Ch_n^2+1-e^{-\lambda h_n}(1+\lambda h_n) \\
&\quad +P(N_\theta((t_{j-1},t_j]\times E)=1 \ {\rm and} \ |{\colorp X_{t_j}^\alpha-\XJ^\alpha}|\leq h_n^\rho {\colorp |X_{t_{j-1}}^\alpha=x_{j-1}})
}
for {\colorp any $x_{j-1}\in K''_{n,j-1}$} and $\alpha\in \Theta$, where {\colorp $C$} does not depend on ${\colorp x_{j-1}}$.
By applying the triangular inequality, we have
\begin{equation*}
|X_{t_j}^\alpha-\XJ^\alpha|\geq |X_{\tau_j}^\alpha-X_{\tau_j-}^\alpha|-|X_{\tau_j-}^\alpha-\XJ^\alpha|-|X_{t_j}^\alpha-X_{\tau_j}^\alpha|,
\end{equation*}
where $\tau_j$ denote the first jump time on $(t_{j-1},t_j]$.
Hence, by using (C4), we obtain
\EQN{\label{p1-diff-est} & P(N_\theta((t_{j-1},t_j]\times E)=1 \ {\rm and} \ |X_{t_j}^\alpha-\XJ^\alpha|\leq h_n^\rho {\colorp |X_{t_{j-1}}^\alpha=x_{j-1}})\\
&\quad\leq P(N_\theta((t_{j-1},t_j]\times E)=1 \ {\rm and} \ |X_{\tau_j-}^\alpha-\XJ^\alpha|+|X_{t_j}^\alpha-X_{\tau_j}^\alpha|>h_n^\rho {\colorp |X_{t_{j-1}}^\alpha=x_{j-1}})\\
&\quad+P(N_\theta((t_{j-1},t_j]\times E)=1 \ {\rm and} \ |X_{\tau_j}^\alpha-X_{\tau_j-}^\alpha|\leq 2h_n^\rho {\colorp |X_{t_{j-1}}^\alpha=x_{j-1}}) \\
&\quad \leq Ch_n^2+\lambda h_n e^{-\lambda h_n}\int_{|z|\leq 2h_n^\rho}F_\theta(z)dz \\
&\quad \leq Ch_n^2+Ch_n^{1+(m+\gamma)\rho}=o(n^{-1})}
{\colorp for $x_{j-1}\in K''_{n,j-1}$, which leads to \eqref{p-diff-est}, and hence} Lemma~\ref{B1-suff-lemma} yields \eqref{tail-prob-est2}.
Thus (B1) holds.

We next check (B2).
For each $l\in\{0,1,2,3\}$, we have
\EQQ{\sup_t \int |\partial_t^l\tilde{p}_j(\alpha_{th})|dx_j
\leq \sup_t \int \{|\PTT^l p_{j,t}^0|1_{L_n}(\Delta x_j)+|\PTT^l p_{j,t}^1|1_{L_n^c}(\Delta x_j)\}dx_j,
}
where $p_{j,t}^l=p_j^l(x_{j-1},x_j,\alpha_{th})$.
Proposition~\ref{pc-est} and Remark~\ref{third-derivative-rem} lead to
\EQQ{\sup_t\int |\PTT^l p_{j,t}^0|dx_j<\infty.}
It follows from Proposition~\ref{pc-est} and (C4) that
{\colorp
\EQN{&\label{p1-int-est} \int |\PTT p_{j,t}^1|dx_j\\
&\quad \leq \int \int_{t_{j-1}}^{t_j}\int \int \bigg|\frac{\PTT p_{x_{j-1},\alpha_{th}}^{c,\tau-t_{j-1}}}{p_{x_{j-1},\alpha_{th}}^{c,\tau-t_{j-1}}}(x)+\frac{\PTT f_{\theta_{th}}}{f_{\theta_{th}}} 1_{\{f_{\theta_{th}}\neq 0\}}(y)+\frac{\PTT p_{x+y,\alpha_{th}}^{c,t_j-\tau}}{p_{x+y,\alpha_{th}}^{c,t_j-\tau}}(x_j)\bigg|p_{x_{j-1},\alpha_{th}}^{c,\tau-t_{j-1}}f_{\theta_{th}}p_{x+y,\alpha_{th}}^{c,t_j-\tau}dxdyd\tau dx_j \\
&\qquad +\frac{Ch_n^2}{\sqrt{T_n}} \\
&\quad\leq Cn^{-1/2}h_n(1+|x_{j-1}|)^C+\frac{Ch_n}{\sqrt{T_n}}\int (1+|y|)^Cf_{\theta_{th}}(y)dy \\
&\qquad +Cn^{-1/2}\int_{t_{j-1}}^{t_j}\int \int (1+|x+y|)^Cp_{x_{j-1},\alpha_{th}}^{c,\tau-t_{j-1}}f_{\theta_{th}}dxdyd\tau+\frac{Ch_n^2}{\sqrt{T_n}} \\
&\quad<\infty,}
where $(\sigma_{th},\theta_{th})=\alpha_{th}$.}
In a similar manner, we can obtain $\sup_t\int |\PTT^l p_{j,t}^1|dx_j<\infty$ for $l\in \{0,1,2,3\}$ {\colorp by Remark~\ref{third-derivative-rem} and (C4)}, and thus \eqref{zeta-integrable} holds.
$D_{j,h}$ can be decomposed as
\EQN{
D_{j,h}= 1-e^{-\lambda_{th}h_n}\left[1+\lambda_{th} h_n\right]+\int(\tilde{p}_j(\alpha_{th})-p_{j,t}^0-p_{j,t}^1)dx_j,}
where {\colorp $\lambda_{th}=\lambda(\theta_{th})$}.
Since
\EQQ{
{\colorp \left|\PTT^l (e^{-\lambda_{th} h_n}\left[1+\lambda_{th} h_n\right])\right|=\bigg|\PTT^{l-1}\bigg(\lambda_{th} h_n^2 e^{-\lambda_{th} h_n}\frac{\PT \lambda(\theta_{th})\cdot h}{\sqrt{T_n}}\bigg)\bigg|=O(h_n^2)}
}
for $l\in\{1,2\}$, H${\rm \ddot{o}}$lder's inequality gives
\EQN{\label{PT-Dj-est} |\PTT^lD_{j,h}|&\leq \left|\PTT^l(e^{-\lambda_{th} h_n}\left[1+\lambda_{th} h_n\right])\right|
+\int |\PTT^l(\tilde{p}_j(\alpha_{th})-p_{j,t}^0-p_{j,t}^1)|dx_j \\
&\quad \leq Ch_n^2+\int \bigg|\DELP{0}{t^l}\bigg|p_{j,t}^01_{L_n^c}dx_j+\int \bigg|\DELP{1}{t^l}\bigg|p_{j,t}^11_{L_n}dx_j \\
&\quad \leq Ch_n^2+\left(\int \bigg|\DELP{0}{t^l}\bigg|^pp_{j,t}^0dx_j\right)^{1/p}\left(\int 1_{L_n^c}p_{j,t}^0dx_j\right)^{1/q} \\
&\quad \quad +\left(\int \bigg|\DELP{1}{t^l}\bigg|^pp_{j,t}^1dx_j\right)^{1/p}\left(\int 1_{L_n}p_{j,t}^1dx_j\right)^{1/q}
}
{\colorp for $x_{j-1}\in K''_{n,j-1}$,} where $p,q>1$ with $1/p+1/q=1$.

{\colorp Proposition~\ref{pc-est}, Jensen's inequality, and a similar argument to (\ref{p1-int-est}) yield}
\EQQ{{\colorp \bigg(\int \bigg|\DELP{k}{t^l}\bigg|^pp_{j,t}^kdx_j\bigg)^{1/p}\leq C_p(1+|x_{j-1}|)^{C_p}\cdot \frac{1}{\sqrt{nh_n^k}}}}
for $k\in \{0,1\}$.
Then as in (\ref{p1-diff-est}), we have
\EQQ{|\PTT^l D_{j,h}|\leq {\colorp o(n^{-1})+Cn^{\delta C_p}(nh_n)^{-1/2}h_n^{(1+(m+\gamma)\rho)/q}
\leq o(n^{-1})+Cn^\epsilon (nh_n)^{-1/2}h_n^{1+(m+\gamma)\rho}}=o(n^{-1})}
{\colorp for $x_{j-1}\in K''_{n,j-1}$ and $q$ satisfying $(1+(m+\gamma)\rho)/q>1+(m+\gamma)\rho-\epsilon/2$,} where $\epsilon$ is a positive constant satisfying $n^{\epsilon+1+(m+\gamma)\rho}=o(1)$
($\delta$ in $K_{n,j}$ should be reset to satisfy $\delta C_p<{\colorp \epsilon/2}$
for $C_p$ and $\epsilon$).
Therefore, we have (\ref{D-del-est}), and hence (B2) holds true.

\subsection{Verifying Conditions (B3) and (B4$'$)}\label{B3B4-subsection}

In this subsection, we look at Conditions (B3) and (B4$'$).
Let $\tilde{f}_t(y)=h_ne^{-\lambda_{th} h_n}f_{\theta_{th}}(y)$.
{\colorp Then} we have
\EQQ{{\colorp p_{j,t}^1=h_n^{-1}\INTTJ \int\int \PCB \tilde{f}_t(y)\PCA dxdyd\tau.}}
{\colorp By Proposition~\ref{gob-prop}, we can see that $p_{j,t}^1>0$ and hence $\tilde{p}_j(\alpha_{th})>0$ for any $x_{j-1},x_j\in \mbbr^m$ and $t\in [0,1]$. Therefore we have}
\EQQ{{\colorp \frac{\PT^l \tilde{p}_j}{\tilde{p}_j}(\alpha_{th})= \DELP{0}{\theta^l}1_{L_n}(\Delta x_j)+\DELP{1}{\theta^l}1_{L_n^c}(\Delta x_j)}.}

For notational simplicity, we write
\EQQ{{\colorp \Phi_j^t(\varphi(x_{j-1},x,y,x_j,\tau))=h_n^{-1}\INTTJ \int\int \varphi(x_{j-1},x,y,x_j,\tau)\PCB \tilde{f}_t(y)\PCA dxdyd\tau}}
for an integrable function $\varphi$.
{\colorp Then we can write 
\EQ{\label{del-p1-eq} \DELP{1}{\theta}=\frac{1}{p_{j,t}^1}\Phi_j^t\bigg(\DELPY{\PT}(y)+\frac{\partial_\theta \PCB}{\PCB}+\frac{\partial_\theta \PCA}{\PCA}\bigg),}
and
\EQN{\label{del2-p1-eq} \DELP{1}{\theta^2}&=\frac{1}{p_{j,t}^1}\Phi_j^t\bigg(\DELPY{\PT^2}(y)+\frac{\PT^2\PCB}{\PCB}+\frac{\PT^2\PCA}{\PCA}+2\frac{\PT \PCB}{\PCB}\DELPY{\PT}(y) \\
&\quad \quad \quad \quad \quad+2\frac{\PT\PCA}{\PCA}\DELPY{\PT}(y)+2\frac{\PT \PCA}{\PCA}\frac{\PT \PCB}{\PCB}\bigg).}
We consider the limit of each term in the right-hand side. 
The following lemma implies that terms related to the derivatives of continuous part are asymptotically negligible.
}

\begin{lemma}\label{del-PCB-conv}
Assume {\colorp (C1)--(C4)}. Then there exists $n_0\in\mbbn$ such that
\EQNN{
\int\bigg|\frac{1}{p_{j,t}^1}\Phi_j^t\bigg(\frac{\PT^l \PCB}{\PCB}\bigg)
1_{L_n^c}\bigg|^{4/l}{\colorp \tilde{p}_j(\alpha_{th})dx_j}&\leq Ch_n^{4-l}(1+|x_{j-1}|)^C, \\
\int\bigg|\frac{1}{p_{j,t}^1}\Phi_j^t\bigg(\frac{\PT^l \PCA}{\PCA}\bigg)
1_{L_n^c}\bigg|^{4/l}{\colorp \tilde{p}_j(\alpha_{th})dx_j}&\leq Ch_n^{4-l}(1+|x_{j-1}|)^C
}
for any $x_{j-1}\in \mbbr^m$, $n\geq n_0$, $t\in[0,1]$, $1\leq j\leq n$, and $l\in\{1,2\}$.
\end{lemma}

\begin{proof}
Jensen's inequality and Proposition~\ref{pc-est} yield that
\EQNN{
&\int\bigg|\frac{1}{p_{j,t}^1}\Phi_j^t\bigg(\frac{\PT^l \PCB}{\PCB}\bigg)
1_{L_n^c}\bigg|^{4/l}\tilde{p}_j (\alpha_{th})dx_j \\
&\quad\leq{\colorp \int \Phi_j^t}\bigg(\bigg|\frac{\PT^l \PCB}{\PCB}\bigg|^{4/l}\bigg)dx_j \\
&\quad ={\colorp \lambda_{th}e^{-\lambda_{th}h_n}}\INTTJ\int\bigg|\frac{\PT^l \PCB}{\PCB}\bigg|^{4/l}\PCB dxd\tau \\
&\quad\leq Ch_n^{4-l}(1+|x_{j-1}|)^C
}
for {\colorp any $x_{j-1}\in \mbbr^m$.}

Similarly, Proposition~\ref{pc-est} {\colorp and (C4) yield}
\EQNN{&\int\bigg|\frac{1}{p_{j,t}^1}\Phi_j^t\bigg(\frac{\PT^l \PCA}{\PCA}\bigg)1_{L_n^c}\bigg|^{4/l}\tilde{p}_j (\alpha_{th})dx_j \\
&\quad \leq {\colorp \int \Phi_j^t}\bigg(\bigg|\frac{\PT^l \PCA}{\PCA}\bigg|^{4/l}\bigg)dx_j \\
&\quad \leq C{\colorp h_n^{2-l}}\INTTJ \int \int (1+|x+y|)^C\PCB \tilde{f}_t(y) dxdyd\tau \\
&\quad \leq Ch_n^{4-l}(1+|x_{j-1}|)^C
}
for {\colorp any $x_{j-1}\in \mbbr^m$.} We also used the fact
\EQQ{\int (1+|x|)^C\PCB dx
=E[(1+|X_{\tau-t_{j-1},x_{j-1}}^{\alpha_{th},c}|)^C]\leq C(1+|x_{j-1}|)^C}
{\colorp by a similar argument to Proposition 3.1 in Shimizu and Yoshida~\cite{shi-yos06}.}
\end{proof}

{\colorp Let $\Delta_j N_\theta=N_\theta((t_{j-1},t_j]\times E)$.} 

\begin{lemma}\label{incprob}
Assume {\colorp (C1)--(C4)}. 
Then, there exist positive constants {\colorp $\iota$ and $C$} such that for all $j\in\{1,\dots,n\}$ and $x_{j-1}\in {\colorp K''_{n,j-1}}$,
\begin{equation}
{\colorp \sup_{\alpha\in \Theta} P \left(X_{t_j}^\alpha-x_{j-1}\in \{z: F_\theta(z)=0\}\cup \{0\}\middle|\Delta_j N_\theta=1, X_{t_{j-1}}^\alpha=x_{j-1}\right)\leq Ch_n^\iota.}
\end{equation}
\end{lemma}

\begin{proof}
{\colorp Let $\rho$ be the one in (C4).}
For the jump time {\colorp $\tau_j$} on {\colorp $(t_{j-1},t_j]$} and large enough $n$, {\colorp (C4) and a similar argument to (\ref{p1-diff-est}) yield}
{\colorp
\begin{align*}
&P \left(X_{t_j}^\alpha-x_{j-1}\in \{z| F_\theta(z)=0\}\cup \{0\}|\D_j N_\theta=1, X_{t_{j-1}}^\alpha=x_{j-1}\right)\\
&\leq P\left(\left\{X_{t_j}^\alpha-x_{j-1}\in\{z| F_\theta(z)=0\}\cup \{0\}\right\}\cap\left\{\left|X_{t_j}^\alpha-X_{\tau_j}^\alpha+X_{\tau_j-}^\alpha-x_{j-1}\right|\leq h_n^\rho\right\}\middle|\D_j N_\theta=1, X_{t_{j-1}}^\alpha=x_{j-1}\right)\\
&+P\left(\left|X_{t_j}^\alpha-X_{\tau_j}^\alpha+X_{\tau_j-}^\alpha-x_{j-1}\right|> h_n^\rho\middle|\D_j N_\theta=1,X_{t_{j-1}}^\alpha=x_{j-1}\right)\\
&\leq P\left(d(X_{\tau_j}^\alpha-X_{\tau_j-}^\alpha,\partial\supp(F_\theta))\leq h_n^\rho\middle|\Delta_jN_\theta=1, X_{t_{j-1}}^\alpha=x_{j-1}\right) \\
&\quad +P\left(|X_{\tau_j}^\alpha-X_{\tau_j-}^\alpha|\leq h_n^\rho\middle|\Delta_jN_\theta=1, X_{t_{j-1}}^\alpha=x_{j-1}\right)
+Ch_n^2 \\
&\leq h_n^{\ep}+Ch_n^{(m+\gamma)\rho}+Ch_n^2.
\end{align*}
}
\end{proof}

We interpret {\colorp $\PT^l\tilde{f}_t/\tilde{f}_t(z)=0$ if $\tilde{f}_t(z)=0$ or $z=0$} for $l\in \{1,2\}$.

\begin{proposition}\label{p1-est-prop}
Assume {\colorp (C1)--(C4)}. Then there exist {\colorp a positive constant $C$,} $n_0\in \mbbn$, and $\epsilon\in (0,1)$ such that
\begin{eqnarray}
\int\bigg|\frac{\PT^l p^1_{j,t}}{p_{j,t}^1}-\frac{\PT^l\tilde{f}_t}{\tilde{f}_t}(\Delta x_j)\bigg|^{4/l}
1_{L_n^c}\tilde{p}_j(\alpha_{th})dx_j 1_{K''_{n,j-1}}(x_{j-1})
&\leq& Ch_n^{1+\ep}(1+|x_{j-1}|)^C, \label{p1-tilde-f-est} \\
\int\bigg|\DELP{1}{\sigma^l}\bigg|^{4/l}1_{L_n^c}\tilde{p}_j(\alpha_{th})dx_j1_{K''_{n,j-1}}(x_{j-1})&\leq& Ch_n(1+|x_{j-1}|)^C, \\
\int\bigg|\DELP{1}{\sigma \PT}\bigg|^21_{L_n^c}\tilde{p}_j(\alpha_{th})dx_j1_{K''_{n,j-1}}(x_{j-1})&\leq& C{\colorp h_n}(1+|x_{j-1}|)^C
\end{eqnarray}
for $t\in [0,1]$, $x_{j-1}\in\mbbr^m$, $l\in \{1,2\}$, {\colorp $1\leq j\leq n$,} and $n\geq n_0$.
\end{proposition}

\begin{proof}
{\colorp Let $Q_j=Q_j(x_{j-1})=\{x_j\in E| \tilde{f}_t(\Delta x_j)\neq0\}\cup \{0\}$.
From Assumption (C4), $Q_j$ does not depend on $t$.}
First we set $l=1$.
We decompose as
\EQNN{&\int\bigg|\frac{\PT^lp^1_{j,t}}{p_{j,t}^1}-\frac{\PT^l\tilde{f}_t}{\tilde{f}_t}(\Delta x_j)\bigg|^{4}1_{L_n^c}\tilde{p}_j(\alpha_{th})dx_j {\colorp 1_{K''_{n,j-1}}}\\
&=\int\bigg\{{\colorp \bigg|\DELP{1}{\theta}-\DELPY{\PT}{\colorp (\Delta x_j)}\bigg|^41_{Q_j}(x_j)+\bigg|\DELP{1}{\theta}\bigg|^41_{Q^c_j}(x_j)\bigg\}}1_{L_n^c}\tilde{p}_j{\colorp (\alpha_{th})dx_j1_{K''_{n,j-1}}}.}
{\colorp (\ref{del-p1-eq}), Jensen's inequality, (\ref{jcond2}),} and Lemma~\ref{del-PCB-conv} yield
\EQN{\label{del-p1-est} &\int\bigg|\DELP{1}{\theta}-\DELPY{\PT}{\colorp (\Delta x_j)}\bigg|^41_{Q_j}(x_j)1_{L_n^c}\tilde{p}_j(\alpha_{th})dx_j1_{K''_{n,j-1}} \\
&\quad \leq {\colorp C}\int \bigg|\frac{1}{p_{j,t}^1}\Phi_j^t\bigg(\DELPY{\PT}(y)-\DELPY{\PT}(\Delta x_j)\bigg)\bigg|^41_{Q_j}(x_j)p_{j,t}^1dx_j1_{K''_{n,j-1}}+Ch_n^3(1+|x_{j-1}|)^C \\
&\quad \leq{\colorp C}\int \Phi_j^t\bigg(\bigg|\DELPY{\PT}(y)-\DELPY{\PT}(\Delta x_j)\bigg|^4\bigg)1_{Q_j}(x_j)dx_j1_{K''_{n,j-1}}+Ch_n^3(1+|x_{j-1}|)^C \\
&\quad \leq C\int \Phi_j^t(|y-\Delta x_j|^4(1+|y|+|\Delta x_j|)^{ C})dx_j1_{K''_{n,j-1}}+Ch_n^3(1+|x_{j-1}|)^C.
}

Obviously, $|y-\Delta x_j|^4\leq C|x_j-x-y|^4+C|x-x_{j-1}|^4$, and we can easily see that
\EQNN{\int |x_j-x-y|^p \PCA dy&=E_{x+y}[|X_{t_j-\tau}^{\alpha_{th},c}-x-y|^p]\leq C(t_j-\tau)^{p/2}(1+|x+y|)^p, \\
\int |x-x_{j-1}|^p \PCB dx&=E_{x_{j-1}}[|X_{\tau-t_{j-1}}^{\alpha_{th},c}-x_{j-1}|^p]\leq C(\tau-t_{j-1})^{p/2}(1+|x_{j-1}|)^p.}
Hence it follows that
\EQN{\label{Phi-Delta-est} &\int \Phi_j^t(|y-\Delta x_j|^4(1+|y|+|\Delta x_j|)^{ C})dx_j1_{K''_{n,j-1}} \\
&\quad \leq Ch_n^{-1}\INTTJ {\colorp \int \int (t_j-\tau)^2}(1+|x+y|)^C\PCB \tilde{f}_t(1+|y|)^C(1+|x+y-x_{j-1}|)^Cdxdyd\tau \\
&\quad \quad +{\colorp C}h_n^{-1}\INTTJ \int \int \PCB \tilde{f}_t|x-x_{j-1}|^4(1+|y|)^{{\colorp C}}(1+|x+y-x_{j-1}|)^{{\colorp C}}dxdyd\tau \\
&\quad \leq  C\INTTJ {\colorp (t_j-\tau)^2} \int \PCB (1+|x|)^C(1+|x-x_{j-1}|)^Cdxd\tau \\
&\quad \quad + C\INTTJ \int \PCB |x-x_{j-1}|^4(1+|x-x_{j-1}|)^{{\colorp C}}dxd\tau \\
&\quad \leq Ch_n^3(1+|x_{j-1}|)^C+ C\INTTJ(\tau-t_{j-1})^2d\tau (1+|x_{j-1}|)^C\\
&\quad \leq Ch_n^3(1+|x_{j-1}|)^C.
}
From the Cauchy--Schwarz inequality, {\colorp (C4), Lemma~\ref{incprob} and a similar estimate to} \eqref{del-p1-est}, we obtain
\EQN{\label{Qjc-4-est}
&\int\bigg|\DELP{1}{\theta}\bigg|^4{\colorp 1_{Q^c_j}(x_j)}1_{L_n^c}\tilde{p}_jdx_j1_{K''_{n,j-1}} \\
&\leq \sqrt{\int\bigg|\DELP{1}{\theta}\bigg|^8{\colorp p_{j,t}^1}dx_j}\sqrt{P\left(\left\{X_{t_j}-x_{j-1}\in{\colorp \{z\in E|\tilde{f}_t(z)=0\}\cup \{0\}}\right\}\cap \left\{ \D_j {\colorp N_\theta}=1\right\} |\XJ=x_{j-1}\right)}1_{K''_{n,j-1}} \\
&\leq C{\colorp h_n^{1+\iota/2}}(1+|x_{j-1}|)^C,
}
so that
\EQQ{\int\bigg|\DELP{1}{\theta}-\DELPY{\PT} (\Delta x_j)\bigg|^41_{L_n^c}\tilde{p}_j(\alpha_{th})dx_j1_{K''_{n,j-1}}(x_{j-1})\leq C {\colorp h_n^{(1+\iota/2)\wedge 3}}(1+|x_{j-1}|)^C.}

Next, we show {\colorp (\ref{p1-tilde-f-est})} for $l=2$.
{\colorp A similar argument to Lemma~\ref{del-PCB-conv} yields}
\EQNN{&\int \bigg|\frac{1}{p_{j,t}^1}\Phi_j^t\bigg(\frac{\PT \mathcal{P}_{j,t}}{\mathcal{P}_{j,t}}\DELPY{\PT}(y)\bigg)\bigg|^21_{L_n^c}\tilde{p}_j{\colorp (\alpha_{th}) dx_j1_{K''_{n,j-1}}} \\
&\quad \leq \sqrt{\int \Phi_j^t\bigg(\bigg|\frac{\PT \mathcal{P}_{j,t}}{\mathcal{P}_{j,t}}\bigg|^4\bigg)dx_j}
\sqrt{\int \Phi_j^t\bigg(\bigg|\DELPY{\PT}(y)\bigg|^4\bigg)dx_j}{\colorp 1_{K''_{n,j-1}}}
\leq C h_n^2(1+|x_{j-1}|)^C,
}
where $\mathcal{P}_{j,t}=\PCB$ or $\PCA$.

Together with Jensen's inequality, {\colorp (\ref{del2-p1-eq}),} Lemma~\ref{del-PCB-conv}, (C4), {\colorp and a similar argument to (\ref{Phi-Delta-est}),} we have
\EQNN{&\BIGINTTP{2}{\DELP{1}{\theta^2}-\DELPY{\PT^2}{\colorp (\Delta x_j)}}{1_{Q_j}(x_j)1_{L_n^c}}{\colorp 1_{K''_{n,j-1}}} \\
&\quad \leq C\int \bigg|\frac{1}{p_{j,t}^1}\Phi_j^t\bigg(\DELPY{\PT^2}(y)-\DELPY{\PT^2}(\Delta x_j)\bigg)\bigg|^21_{Q_j}(x_j)p_{j,t}^1dx_j{\colorp 1_{K''_{n,j-1}}}
+Ch_n^2(1+|x_{j-1}|)^C \\
&\quad \leq C\int\Phi_j^t(|y-\Delta x_j|^2(1+|y|+|\Delta x_j|)^C)dx_j
+Ch_n^2(1+|x_{j-1}|)^C \\
&\quad \leq Ch_n^2(1+|x_{j-1}|)^C.
}
{\colorp Together with a similar argument to (\ref{Qjc-4-est}), we obtain (\ref{p1-tilde-f-est}) for $l=2$.}

For the estimate for $\PS^lp_{j,t}^1$, we first have
\EQNN{\DELP{1}{\sigma}&=\frac{1}{p_{j,t}^1}\Phi_j^t\bigg(\frac{\PS \PCA}{\PCA}+\frac{\PS \PCB}{\PCB}\bigg), \\
\DELP{1}{\sigma^2}&=\frac{1}{p_{j,t}^1}\Phi_j^t\bigg(\frac{\PS^2 \PCA}{\PCA}+\frac{\PS^2 \PCB}{\PCB}+2\frac{\PS \PCB}{\PCB}\frac{\PS \PCA}{\PCA}\bigg).
}
{\colorp Thanks to} Jensen's inequality and Proposition~\ref{pc-est},
we obtain
\EQNN{&\BIGINTTP{4}{\DELP{1}{\sigma}}{1_{L_n^c}}1_{K''_{n,j-1}} \\
&\quad \leq C\int \Phi_j^t\bigg(\bigg|\frac{\PS \PCA}{\PCA}+\frac{\PS \PCB}{\PCB}\bigg|^4\bigg)dx_j1_{K''_{n,j-1}} \\
&\quad \leq Ch_n(1+|x_{j-1}|)^C
}
and
\EQNN{&\BIGINTTP{2}{\DELP{1}{\sigma^2}}{1_{L_n^c}}1_{K''_{n,j-1}} \\
&\quad \leq C\int \Phi_j^t\bigg(\bigg|\frac{\PS^2 \PCA}{\PCA}+\frac{\PS^2 \PCB}{\PCB}+2\frac{\PS \PCB}{\PCB}\frac{\PS \PCA}{\PCA}\bigg|^2\bigg)dx_j1_{K''_{n,j-1}} \\
&\quad \leq Ch_n(1+|x_{j-1}|)^C.
}

{\colorp Similarly, we have}
\EQQ{\BIGINTTP{2}{\DELP{1}{\sigma \PT}}{1_{L_n^c}}1_{K''_{n,j-1}}\leq C{\colorp h_n}(1+|x_{j-1}|)^C.}

\end{proof}

\begin{proposition}
Assume {\colorp (C1)--(C4)}. Then (B3) holds true.
\end{proposition}

\begin{proof}
{\colorp (\ref{dj-est}),} Propositions~\ref{pc-est}, and \ref{p1-est-prop} yield
\EQNN{&\SUMJ\sup_{t\in[0,1]}D_{j,h}^{-1}\int (|\zeta_{j,t}^{2,h}|^2+|\zeta_{j,t}^{1,h}|^4)\tilde{p}_j(\alpha_{th})dx_j{\colorp 1_{K''_{n,j-1}}} \\
&\quad \leq C\SUMJ\sup_{t\in[0,1]}D_{j,h}^{-1}\int\bigg(\bigg|\frac{1}{n}\DELTP{\PS^2}\bigg|^2+\bigg|\frac{1}{T_n}\DELTP{\PT^2}\bigg|^2+\bigg|\frac{1}{\sqrt{nT_n}}\DELTP{\PT\PS}\bigg|^2+\frac{1}{n^2}\bigg|\DELTP{\PS}\bigg|^4+\frac{1}{T_n^2}\bigg|\DELTP{\PT}\bigg|^4\bigg)\tilde{p}_j(\alpha_{th})dx_j1_{K''_{n,j-1}} \\
&\quad \leq \frac{C}{T_n^2}\SUMJ\sup_{t\in[0,1]}D_{j,h}^{-1}\int \bigg(\bigg|\DELPY{\PT^2}\bigg|^2+\bigg|\DELPY{\PT}\bigg|^4\bigg)(\Delta x_j)1_{L_n^c}\tilde{p}_j(\alpha_{th})dx_j1_{K''_{n,j-1}}+o_p(1).
}
Then {\colorp (\ref{dj-est}), (B1), (C4) and a similar argument to (\ref{Phi-Delta-est})} yield the conclusion because $\PT^2\tilde{f}_t/\tilde{f}_t=\PT^2\log \tilde{f}_t+(\PT\log \tilde{f}_t)^2$.

\end{proof}

We turn to observe (B4$'$). Let $\tilde{p}_{j,0}=\tilde{p}_j(\alpha_0)$.
{\colorp (B1), (\ref{dj-est}), Proposition~\ref{p1-est-prop}, and the Cauchy--Schwarz inequality} yield
\EQNN{&\SUMJ D_{j,0}^{-1}\INTTP{0}{\tilde{\eta_j}\tilde{\eta}_j^\top} \\
&\quad =\SUMJ D_{j,0}^{-1}\INTTP{0}{
\MAT{cc}{{\colorp \frac{\PS p_{j,0}^0(\PS p_{j,0}^0)^\top}{n(p_{j,0}^0)^2}}1_{L_n}
& \frac{\PS p_{j,0}^0(\PT p_{j,0}^0)^\top}{n\sqrt{h_n}(p_{j,0}^0)^2}1_{L_n} \\
\frac{\PT p_{j,0}^0(\PS p_{j,0}^0)^\top}{n\sqrt{h_n}(p_{j,0}^0)^2}1_{L_n}
& \frac{\PT\tilde{f}_0\PT \tilde{f}_0^\top}{nh_n\tilde{f}_0^2}{\colorp (\Delta x_j)}1_{L_n^c}+\frac{\PT p_{j,0}^0(\PT p_{j,0}^0)^\top}{nh_n(p_{j,0}^0)^2}1_{L_n}}}+o_p(1).
}
\bigskip

{\colorp Hence,} 
(B4$'$) follows if we show
\EQ{\label{ps-conv} \frac{1}{n}\SUMJ D_{j,0}^{-1}\INTTP{0}{\DELPPT{\PS}{\PS}1_{L_n}}\TOP  \Gamma_1,}
\EQ{\label{pt-conv} \frac{1}{nh_n}\SUMJ D_{j,0}^{-1}\INTTP{0}{\bigg\{\DELPYYT{\colorp (\Delta x_j)}1_{L_n^c}+\DELPPT{\PT}{\PT}1_{L_n}\bigg\}}\TOP  \Gamma_2,}
\EQ{\label{pspt-conv} \frac{1}{n\sqrt{h_n}}\SUMJ D_{j,0}^{-1}\INTTP{0}{\DELPPT{\PS}{\PT}1_{L_n}}\TOP 0.}

\begin{lemma}\label{phi-p0-lemma}
Assume (C1) and (C2). Then there {\colorp exist a positive constant $C$ and} a sequence $({\colorp \chi_n})_{n=1}^\infty$ such that ${\colorp \chi_n}\to 0$ as $n\to \infty$ and
\EQNN{
\INTTP{0}{\bigg|\DELPP{0}{\sigma}+\frac{1}{2}\Delta x_j^\top {\colorp \mathfrak{S}_j}h_n^{-1}\Delta x_j-\frac{1}{2}{\rm tr} ({\colorp \mathfrak{S}_jS_j})\bigg|^21_{L_n}}
&\leq C{\colorp \chi_n} (1+|x_{j-1}|)^C, \\
\INTTP{0}{\bigg|\DELPP{0}{\theta}-{\colorp \partial_\theta a^\top (x_{j-1},\theta_0) S_j^{-1}}\Delta x_j\bigg|^21_{L_n}}
&\leq C h_n{\colorp \chi_n} (1+|x_{j-1}|)^C
}
for any $x_{j-1}\in {\colorp \mbbr^m}$, {\colorp where $\mathfrak{S}_j=\partial_\sigma S^{-1}(x_{j-1},\sigma_0)$ and $S_j=S(x_{j-1},\sigma_0)$.}
\end{lemma}

\begin{proof}
(3.23) and (3.29) in Gobet~\cite{gob02} yield
{\colorp 
\EQNN{\delta(( U^{\alpha,h_n})^\top \PT X_{h_n,x}^{\alpha,c})
&=h_n(\PT a(x,\theta))^\top S^{-1}(x,\sigma)(X_{h_n,x}^{\alpha,c}-x)+ H_{1,x}, \\
\delta((U^{\alpha,h_n})^\top \PS X_{h_n,x}^{\alpha,c})
&=-\frac{1}{2}(X_{h_n}^{\alpha,c}-x)^\top \PS S^{-1}(x,\sigma)(X_{h_n,x}^{\alpha,c}-x)
+\frac{1}{2}h_n{\rm tr}(\PS S^{-1}(x,\sigma) S(x, \sigma))+ H_{2,x}
}}
{\colorp in the setting of Section~\ref{conti-part-subsection}, where $E[ |H_{1,x}|^p]^{1/p}\leq Ch_n^{3/2}\chi_n(1+|x|)^C$
and $E[ |H_{2,x}|^p]^{1/p}\leq Ch_n\chi_n(1+|x|)^C$ for any $p\geq 2$.}
Hence Proposition~\ref{gob-prop2} implies that
\EQNN{
& \INTTP{0}{\bigg|\DELPP{0}{\sigma}+\frac{1}{2}\Delta x_j^\top \mathfrak{S}_jh_n^{-1}\Delta x_j-\frac{1}{2}{\rm tr} (\mathfrak{S}_jS_j)\bigg|^21_{L_n}}\\
&= {\colorp h_n^{-2}}\INTTP{0}{\Big|E\left[H_{2,x_{j-1}}\big| {\colorp X^{\alpha_0,c}_{h_n,x_{j-1}}}=x_j\right]
\Big|^21_{L_n}}\\
&\leq {\colorp h_n^{-2}} E\left[\left|H_{2,x_{j-1}}\right|^2\right]
\leq {\colorp C \chi_n^2}(1+|x_{j-1}|)^C
}
for any $x_{j-1}\in {\colorp \mbbr^m}$.

The other formula can be shown in the same way.

\end{proof}

\begin{lemma}\label{yu:sumnor}
Assume (C1)--(C4). Let $R$ be a differentiable function on $\mbbr^m$ satisfying
\EQ{\label{R-ineq} |\partial_x^lR(x)|\leq C(1+|x|)^C}
for any $x\in\mbbr^m$ and $l\in\{0,1\}$ with some positive constant $C$. Then 
\begin{equation}
\frac{1}{n}\sumj D_{j,0}^{-1}(x_{j-1},0) R(x_{j-1})\TOP \int R(x)\pi(dx).
\end{equation}
\end{lemma}
\begin{proof}
(\ref{R-ineq}), (C3), (\ref{dj-est}), {\colorp and a suitable choice of $\delta$ in $K_{n,j}$} yield
{\colorp 
\EQNN{\left|\frac{1}{n}\sumj (1-D_{j,0}^{-1}) R(x_{j-1})1_{\cap_{j'}K_{n,j'}}\right| 
&\leq \delta_n\frac{1}{n}\sumj |R(x_{j-1})| \\
&=o_p(n^{-1})\cdot n^{\delta C}\TOP 0.}
}
Together with {\colorp (\ref{tail-prob-est2})}, we have $n^{-1}\sum_j(1-D_{j,0}^{-1})R(x_{j-1})\TOP 0$.
Moreover, (C3) and (\ref{R-ineq}) yield $n^{-1}\sum_jR(x_{j-1})\to^{P_{\alpha_0,n}} \int R(x)\pi(dx)$.
Then, Theorem~\ref{log-likelihood-approx-thm} yields the desired result.
\end{proof}

Let {\colorp $\Delta X^c_j=X_{h_n,x_{j-1}}^{\alpha_0,c}-x_{j-1}$}. Since
\EQQ{{\colorp \int \Delta x_j1_{L_n}\tilde{p}_{j,0}dx_j=e^{-\lambda h_n}(E[\Delta X^c_j]-E[\Delta X^c_j1_{L_n^c}(\Delta X^c_j)])=o(\sqrt{h_n}),}}
\EQQ{\INTTP{0}{\Delta x_j\Delta x_j^\top1_{L_n}}
={\colorp e^{-\lambda h_n}(E\left[(\Delta X^c_j)(\Delta X^c_j)^\top\right] -E\left[(\Delta X^c_j)(\Delta X^c_j)^\top1_{L_n^c}(\Delta X^c_j)\right])}
=h_n S_j+o(h_n),}
and
\EQNN{&\INTTP{0}{[\Delta x_j]_{i_1}[\Delta x_j]_{i_2}[\Delta x_j]_{i_3}[\Delta x_j]_{i_4}1_{L_n}}\\
&={\colorp e^{-\lambda h_n}}{\colorp E}\left[[\Delta X^c_j]_{i_1}[\Delta X^c_j]_{i_2}[\Delta X^c_j]_{i_3}[\Delta X^c_j]_{i_4}\right] \\
&\quad -{\colorp e^{-\lambda h_n}E}\left[[\Delta X^c_j]_{i_1}[\Delta X^c_j]_{i_2}[\Delta X^c_j]_{i_3}[\Delta X^c_j]_{i_4}1_{L_n^c}(\Delta X^c_j)\right] \\
&=h_n^2([S_j]_{i_1i_2}[S_j]_{i_3i_4}+[S_j]_{i_1i_3}[S_j]_{i_2i_4}+[S_j]_{i_1i_4}[S_j]_{i_2i_3})+o(h_n^2)
}
for $x_{j-1}\in K''_{n,j-1}$ {\colorp with suitable choice of $\delta$ in $K_{n,j-1}$},
the following lemma completes the proof of Theorem~\ref{jump-LAN-thm}.

\begin{lemma}
Assume (C1)--(C4). Then
\EQ{\label{ps-conv2} \frac{1}{2n}\SUMJ D_{j,0}^{-1}{\colorp \left({\rm tr}(S_j^{-1}\partial_{\sigma_k} S_j S_j^{-1}\partial_{\sigma_l} S_j)\right)_{1\leq k,l\leq d_1}}\TOP \Gamma_1,}
\EQ{\label{pt-conv2} \frac{1}{nh_n}\SUMJ D_{j,0}^{-1}\bigg({\colorp \mathfrak{A}_{j,k}^\top S_j^{-1} \mathfrak{A}_{j,l}h_n}+\INTTP{0}{\frac{\partial_{\theta_k}\tilde{f}_0\partial_{\theta_l}\tilde{f}_0}{\tilde{f}_0^2}{\colorp (\Delta x_j)}1_{L_n^c}}\bigg)_{1\leq k,l\leq d_2}\TOP  \Gamma_2,}
\EQ{\label{pspt-conv2} \frac{1}{nh_n^{3/2}}\SUMJ D_{j,0}^{-1}\INTTP{0}{{\colorp \mathfrak{A}_{j,k'}^\top S_j^{-1}\Delta x_j\Delta x_j^\top \partial_{\theta_{l'}} S_j^{-1}(x_{j-1},\sigma_0)}\Delta x_j1_{L_n}}\TOP 0}
for $1\leq k'\leq d_1$ and $1\leq l'\leq d_2$, {\colorp where $\mathfrak{A}_{j,k}=\partial_{\theta_k}a(x_{j-1},\theta_0)$.}
\end{lemma}

\begin{proof}
Since
\EQQ{ {\colorp \bigg|\INTTP{0}{[\Delta x_j]_{i_1}[\Delta x_j]_{i_2}[\Delta x_j]_{i_3}1_{L_n}}\bigg|\leq |E[[X_{h_n,x_{j-1}}^{\alpha_0,c}]_{i_1}[X_{h_n,x_{j-1}}^{\alpha_0,c}]_{i_2}[X_{h_n,x_{j-1}}^{\alpha_0,c}]_{i_3}1_{L_n}(\D X^c_j)]|\leq Ch_n^2(1+|x_{j-1}|)^C}}
for $x_{j-1}\in K''_{n,j-1}$, (\ref{pspt-conv2}) holds true by {\colorp (\ref{dj-est}) and (B1)}.

Moreover, thanks to Lemma~\ref{yu:sumnor},
we have (\ref{ps-conv2}) and
\EQQ{\frac{1}{ n}\SUMJ D_{j,0}^{-1}{\colorp \mathfrak{A}_{j,k}^\top S_j^{-1}\mathfrak{A}_{j,l}}\TOP \int \partial_{\theta_k} a^\top S^{-1}\partial_{\theta_l} a(x,\alpha_0)d\pi(x)}
for $1\leq k,l\leq d_2$.
Then it is sufficient to show that
\EQ{\label{pt-conv3} \frac{1}{nh_n}\SUMJ D_{j,0}^{-1}\INTTP{0}{\frac{\partial_{\theta_k}\tilde{f}_0\partial_{\theta_l}\tilde{f}_0}{\tilde{f}_0^2}{\colorp (\Delta x_j)}1_{L_n^c}}
\TOP  \int \frac{\partial_{\theta_k}f_{\theta_0}\partial_{\theta_l}f_{\theta_0}}{f_{\theta_0}} 1_{\{f_{\theta_0}\neq 0\}}(y)dy.
}

Similar arguments to (\ref{del-p1-est}), {\colorp (\ref{Phi-Delta-est}), (\ref{p1-diff-est}), and (\ref{PT-Dj-est})}  yield
\EQNN{& \frac{1}{D_{j,0}}\INTTP{0}{\bigg(\DELPYY{\PT}\bigg)^2{\colorp (\Delta x_j)}1_{L_n^c}} \\
&\quad =\frac{1}{D_{j,0}}\int \bigg(\DELPYY{\PT}\bigg)^2(\Delta x_j){\colorp p_{j,0}^1}1_{L_n^c}dx_j \\
&\quad=\frac{1}{D_{j,0}}\bigg\{\int \Phi_j^0\bigg(\bigg(\DELPYY{\PT}\bigg)^2(y)\bigg)dx_j+\int \Phi_j^0\bigg(\bigg(\DELPYY{\PT}\bigg)^2(\Delta x_j)-\bigg(\DELPYY{\PT}\bigg)^2(y)\bigg)dx_j \\
&\quad \quad \quad \quad +\int {\colorp \bigg(\DELPYY{\PT}\bigg)^2(\Delta x_j)p_{j,0}^1}1_{L_n}dx_j\bigg\} \\
&\quad =\frac{ 1}{D_{j,0}}\int \bigg(\DELPYY{\PT}\bigg)^2 \tilde{f}_0(y)dy+ {\colorp o(h_n)} \\
&\quad =h_n\int \frac{(\PT f_{\theta_0})^2}{f_{\theta_0}} 1_{\{f_{\theta_0}\neq 0\}}(y)dy+ o(h_n)
}
{\colorp for $x_{j-1}\in K''_{n,j-1}$ with suitable choice of $\delta$ in $K_{n,j-1}$. Here we used (\ref{dj-est})} and that
$\DELPYY{\PT}=-h_n\PT \lambda(\theta_0) +\frac{\PT f_{\theta_0}}{f_{\theta_0}}$ if $f_{\theta_0}\neq 0$. 
Therefore, we have (\ref{pt-conv3}).

\end{proof}

\begin{small}
\bibliographystyle{abbrv}
\bibliography{referenceLibrary_mathsci,referenceLibrary_other}
\end{small}

\end{document}